\documentclass[12pt]{article}
\usepackage{geometry}                % See geometry.pdf to learn the layout options. There are lots.
\usepackage{graphicx}
\usepackage{amssymb}
\usepackage{amsmath}
\usepackage{mathtools}
\usepackage{mathrsfs}
\usepackage{xcolor}
\usepackage{amsthm}
\usepackage{hyperref}
\usepackage{pstricks}
\usepackage[backend=bibtex, style=alphabetic]{biblatex}
\AtEveryBibitem{
    \clearfield{urlyear}
    \clearfield{urlmonth}
    \clearfield{doi}
    \clearfield{isbn}
    \clearfield{issn}
    \clearfield{pagetotal}
}

\newcommand*{\FF}{\mathbb{F}}
\newcommand*{\NN}{\mathbb{N}}
\newcommand*{\ZZ}{\mathbb{Z}}

\newcommand*{\RR}{\mathbb{R}}

\newcommand*{\PP}{\mathbb{P}}
\DeclareMathOperator*{\EE}{\mathbb{E}}

\newcommand*{\calE}{\mathcal{E}}
\newcommand*{\calF}{\mathcal{F}}

\newcommand*{\calO}{\mathcal{O}}
\newcommand*{\calP}{\mathcal{P}}
\newcommand*{\calU}{\mathcal{U}}
\newcommand*{\calV}{\mathcal{V}}
\newcommand*{\calW}{\mathcal{W}}

\newcommand*{\tail}{\mathscr{T}}

\newcommand*{\Gibbs}{\mathscr{G}}

\newcommand*{\A}{\mathtt{A}}
\newcommand*{\ta}{\mathtt{a}}
\newcommand*{\B}{\mathtt{B}}
\newcommand*{\tb}{\mathtt{b}}

\newcommand*{\1}{\mathbf{1}}

\newcommand*{\mb}[1]{\mathbf{#1}}
\newcommand*{\bOmega}{\mb{\Omega}}

\newcommand*{\st}{\,:\,}

\newcommand*{\join}{\mathcal{J}}

\newcommand*{\mcut}{\mathsf{mcut}}
\newcommand*{\Pstar}{\mathsf{P}_*}

\newcommand*{\diag}{\mathbin{\vartriangle}}

\newcommand*{\ball}[3][\relax]{\mathrm{B}^{#1}(#2, #3)}

\newtheorem{lemma}{Lemma}
\newtheorem{cor}[lemma]{Corollary}
\newtheorem{prop}[lemma]{Proposition}
\newtheorem{theorem}[lemma]{Theorem}

\newtheorem{mainthm}{Theorem}
\newtheorem{mainthmstate}{Theorem}

\theoremstyle{definition}

\DeclareMathOperator{\Prob}{Prob}

\DeclareMathOperator{\Unif}{Unif}

\DeclareMathOperator{\Hom}{Hom}

\DeclareMathOperator{\Sym}{Sym}

\DeclareMathOperator{\Lip}{Lip}
\DeclareMathOperator{\BL}{BL}

\DeclareMathOperator{\diam}{diam}
\DeclareMathOperator{\shent}{H}
\DeclareMathOperator{\h}{h}
\DeclareMathOperator{\f}{f}
\newcommand*{\fe}{A}

\DeclareMathOperator{\ex}{ex}

\newcommand{\Is}[1]{\mathsf{is}_{#1}}

\DeclarePairedDelimiter{\abs}{\lvert}{\rvert}
\DeclarePairedDelimiter{\floor}{\lfloor}{\rfloor}

\DeclarePairedDelimiter{\norm}{\|}{\|}
\newcommand{\nnorm}[1]{{\left\vert\kern-0.25ex\left\vert\kern-0.25ex\left\vert #1 
    \right\vert\kern-0.25ex\right\vert\kern-0.25ex\right\vert}}
\DeclarePairedDelimiterX{\inprod}[2]{\langle}{\rangle}{#1,\ #2}

\title{Metastability and maximal-entropy joinings of Gibbs measures on finitely-generated groups}
\author{Christopher Shriver}

\begin{document}
\maketitle

\begin{abstract}
	We prove a metastability result for finitary microstates which are good models for a Gibbs measure for a nearest-neighbor interaction on a finitely-generated group. This is used to show that any maximal-entropy joining of two such Gibbs states is a relative product over the tail $\sigma$-algebra, except in degenerate cases.
	
	We also use results on extremal cuts of random graphs to further investigate optimal self-joinings of the Ising model on a free group.
\end{abstract}

\section{Introduction, main results}

Let $\Gamma$ be a countably infinite group with $r$ generators $s_1, \ldots, s_r$, and let $\A$ be a finite set. We will also use $\Gamma$ to denote the left Cayley graph of the group, which has vertex set $\Gamma$ and an $s_i$-labeled directed edge $(\gamma, s_i \gamma)$ for each $i \in [r] = \{1, 2, \ldots, r\}$.

The group $\Gamma$ acts on itself by right multiplication; note that this action consists of isomorphisms of the Cayley graph which preserve edge labels and directions. We also let $\Gamma$ act on the set of labelings $\A^\Gamma$: given $\mb{x} \in \A^\Gamma$ and $\beta \in \Gamma$, the shifted labeling $\beta\mb{x}$ is given by
	\[ \big(\beta\mb{x}\big) (\gamma) = \mb{x}(\gamma \beta) . \]
This also induces an action on $\Prob(\A^\Gamma)$ by pushforwards. A probability measure invariant under this action will be called shift-invariant; the set of such measures will be denoted $\Prob^\Gamma(\A^\Gamma)$. \\

We will think of a measure $\mu \in \Prob^\Gamma(\A^\Gamma)$ as specifying local statistics of finite systems according to the following paradigm:

Given a finite set $V$ and a homomorphism $\sigma \in \Hom(\Gamma, \Sym(V))$, we can construct a multigraph with an $s_i$-labeled directed edge $(v, \sigma^{s_i} v)$ for each $v \in V$ and $i \in [r]$; this will be called the graph of $\sigma$.

If $\mb{x} \in \A^V$ is any labeling of $V$ by elements of $\A$, we can pull back $\mb{x}$ to a labeling of $\Gamma$. This is called a pullback name of $\mb{x}$, and is denoted
	\[ \Pi_v^\sigma \mb{x} \coloneqq \big( \mb{x}(\sigma^\gamma v) \big)_{\gamma \in \Gamma} \in \A^\Gamma . \]

The \emph{empirical distribution} of $\mb{x}$ over $\sigma$ is the distribution of these pullback names if the basepoint $v$ is chosen uniformly at random:
	\[ P_{\mb{x}}^\sigma \coloneqq \big( v \mapsto \Pi_v^\sigma \mb{x})_* \Unif(V) = \frac{1}{\abs{V}} \sum_{v \in V} \delta_{\Pi_v^\sigma \mb{x}} \in \Prob^\Gamma(\A^\Gamma) . \]
The shift-invariance of every empirical distribution is the reason we assumed $\mu$ above was shift-invariant.

By analogy with statistical physics we will call $\mb{x}$ a microstate, and we will call it a good model for $\mu$ if its empirical distribution (over some given $\sigma$) is close to $\mu$. More specifically, if $\calO$ is some weak-open neighborhood of $\mu$ then we say $\mb{x}$ is an $\calO$-microstate if $P_{\mb{x}}^\sigma \in \calO$. We call the set of such $\mb{x}$
	\[ \Omega(\sigma, \calO) = \{ \mb{x} \in \A^V \st P_{\mb{x}}^\sigma \in \calO \} . \]
This is equivalent to Lewis Bowen's framework of ``approximating partitions'' introduced in \cite{bowen2010b} to define sofic entropy. We will discuss entropy below. \\

This notion of ``good model'' is most meaningful when the graph of $\sigma$ has a high degree of local similarity to $\Gamma$. We will measure this in the following way: given $R \in \NN$, define
	\[\delta_R^\sigma = \frac{1}{V} \abs{\{ v \in V \st \ball[\sigma]{v}{R} \not\cong \ball[\Gamma]{e}{R} \}} . \]
Here the isomorphism is between the subgraphs induced by the radius-$R$ balls centered at $v$ in the graph of $\sigma$ and those centered at the identity in the Cayley graph of $\Gamma$. Recall that we consider edges of the graph of $\sigma$ and of the Cayley graph to be directed and labeled by the generators of $\Gamma$; we require isomorphisms to respect this structure.

We then make the slightly more ad hoc definition
	\[ \Delta^\sigma = \inf_R \big(9 \cdot (2/3)^{R} + 6\delta_R^\sigma \big). \]
The particular constants appearing here come from our choice of metric on $\A^\Gamma$ (see Section \ref{sec:definitions}) and from some details of proofs in \cite{shriver2020a}. If $\Delta^\sigma$ is small, then the graph of $\sigma$ looks like $\Gamma$ to a large radius near most vertices. Note that the notation $\Delta^\sigma$ does not need to explicitly specify which $\Gamma$ the graph of $\sigma$ is being compared to, since the relevant $\Gamma$ is always the domain of $\sigma$.

Let $\Sigma = (\sigma_n \in \Hom(\Gamma, \Sym(V_n))_{n \in \NN}$ be a sequence of homomorphisms, with $V_n$ finite sets. We call $\Sigma$ a \emph{sofic approximation} to $\Gamma$ if $\lim_{n \to \infty} \Delta^{\sigma_n} = 0$. The sofic entropy of $\mu \in \Prob^\Gamma(\A^\Gamma)$ relative to $\Sigma$ is defined by
	\[ \h_\Sigma(\mu) = \inf_{\calO \ni \mu} \limsup_{n \to \infty} \frac{1}{\abs{V_n}} \log \abs{\Omega(\sigma_n, \calO)} , \]
where the infimum is over weak-open neighborhoods of $\mu$. Informally, we would expect $\abs{\Omega(\sigma_n, \calO)}$ to grow exponentially with $\abs{V_n}$, with a higher exponential growth rate indicating fewer constraints imposed by $\mu$ on its good models (so $\mu$ is ``more random''). In general, though, sofic entropy may behave in counterintuitive ways. While it is an isomorphism invariant, an example of Ornstein and Weiss \cite{ornstein1987} shows that it may increase under factor maps when $\Gamma$ is not amenable.

The assumption $\Delta^{\sigma_n} \to 0$ is interpreted here as a kind of Benjamini-Schramm convergence, but we can also view it as requiring the actions $\Gamma \curvearrowright^{\sigma_n} V_n$ to be ``asymptotically free.'' More generally we could only require that they be ``asymptotically actions'' (see for example \cite{bowen2020a}) but for simplicity we only consider true homomorphisms here. \\

In the present paper we restrict attention to measures $\mu$ which are Gibbs for some nearest-neighbor interaction; relevant definitions are given in Section \ref{sec:definitions}. For a nearest-neighbor interaction $\Phi$, we denote the set of Gibbs measures by $\Gibbs(\Phi) \subset \Prob(\A^\Gamma)$. The set of shift-invariant Gibbs measures is denoted $\Gibbs^\Gamma(\Phi)$. An interaction also comes with an associated ``Glauber dynamics'' which is a natural and useful model for the random evolution of a system over time. We will use subscripts to denote evolution under Glauber dynamics; for example $\mb{x}_s$ is the (random) evolution of the microstate $\mb{x}_0$.

Our first main result, Theorem \ref{thm:Gibbsmetastability}, establishes the metastability of Gibbs microstates under Glauber dynamics: %if a microstate $\mb{x} \in \A^V$ has empirical distribution (over some $\sigma \in \Hom(\Gamma, \Sym(V))$) near a Gibbs state $\mu$ then, as $\mb{x}$ evolves under the Glauber dynamics, its empirical distribution stays close to $\mu$ for a long time with high probability. How long and how close it stays to $\mu$, and with what probability, is controlled by $\Delta^\sigma$ (i.e. amount of local similarity to $\Gamma$) and how close it starts to $\mu$.

\begin{mainthmstate}
	Let $\mu \in \Gibbs^\Gamma(\Phi)$ for some nearest-neighbor interaction $\Phi$. Denote its evolution under the Glauber dynamics for $\Phi$ as $\{ \mu_t \st t \geq 0 \}$. 
	
	Given any neighborhood $\calU_1$ of $\mu$ and $t,\varepsilon>0$, there exists a neighborhood $\calU_0$ of $\mu$ and $\delta>0$ such that, for any finite set $V$ and any homomorphism $\sigma \colon \Gamma \to \Sym(V)$, if $\mb{x}_0 \in \Omega(\sigma,\calU_0)$ and $\Delta^\sigma < \delta$ then $\mb{x}_s \in \Omega(\sigma, \calU_1)$ for all $s \in [0,t]$ with probability at least $1-\varepsilon$.
\end{mainthmstate}

We call this ``metastability'' because, if we let the Glauber dynamics run forever, the law of $\mb{x}$ will converge to the (unique) Gibbs measure on $\A^V$. In particular, we will eventually lose control of its empirical distribution. Theorem \ref{thm:Gibbsmetastability} only says that for any \emph{fixed} time $t$, it can be arranged for the empirical distribution to stay close to $\mu$ for time $t$ with probability as close to 1 as desired. The only requirements are that $\Delta^\sigma$ be small enough and that $P_{\mb{x}}^\sigma$ start close enough to $\mu$.

The main technical result of \cite{shriver2020a} (repeated below as Theorem \ref{thm:Sequivariance}) is a type of equivariance between the Glauber dynamics on $\Gamma$ and on graphs of homomorphisms $\sigma$ with small $\Delta^\sigma$: it implies that if $\mb{x}$ is a good model for $\mu$ (not necessarily Gibbs) then the \emph{expected} empirical distribution of the evolved microstate $\mb{x}_t$ stays close to the evolved measure $\mu_t$. The rate at which it drifts away is controlled by $\Delta^\sigma$. But if $\mu$ is Gibbs then it is Glauber-invariant, so in fact the expected empirical distribution stays close to $\mu$.

It turns out to be somewhat difficult to conclude that the empirical distribution actually stays close to $\mu$ with high probability. We do this in two steps: first we use the fact that the Gibbs measures form a face of the convex set $\Prob^\Gamma(\A^\Gamma)$,  combined with the mentioned equivariance result, to show that the empirical distribution of $\mb{x}$ stays approximately Gibbs for the desired amount of time with high probability. We then use this approximate Gibbs-ness to show that the empirical distribution tends to move slowly, so typically stays close $\mu$. \\

Using Theorem \ref{thm:Gibbsmetastability} we establish Theorem \ref{thm:gibbsmaximal}, which says that any maximal-entropy joining of two Gibbs measures (possibly for different interactions) must itself be a Gibbs measure for a natural ``sum interaction,'' except in degenerate cases:

\begin{mainthmstate}
	Let $\lambda$ be a joining of two Gibbs measures $\mu_\A \in \Prob(\A^\Gamma), \mu_\B \in \Prob(\B^\Gamma)$ for nearest-neighbor interactions $\Phi^\A, \Phi^\B$ respectively. Let $\Sigma$ be a random sofic approximation to $\Gamma$, and assume that there is some joining $\lambda$ of $\mu_\A, \mu_\B$ with $\h_\Sigma(\lambda) > -\infty$.
	
	If $\lambda$ maximizes $\h_\Sigma$ among all joinings of $\mu_\A, \mu_\B$, then $\lambda \in \Gibbs^\Gamma (\Phi^\A \oplus \Phi^\B)$.
\end{mainthmstate}

Here, a random sofic approximation is a sequence of random homomorphisms such that for any $\delta >0$ the probability of the event $\{\Delta^{\sigma_n} < \delta\}$ approaches 1 superexponentially fast; see Section \ref{sec:application}. The $\f$-invariant, introduced in \cite{bowen2010a}, can be written as the sofic entropy relative to a random sofic approximation to a free group \cite{bowen2010}.

By \cite[Equation (7.19)]{georgii2011}, we can equivalently say that a maximal-entropy joining of two Gibbs measures must be a relative product over the tail $\sigma$-algebra.

We also mention two brief corollaries: Corollary \ref{cor:prodmax} shows that if $\mu_\A$ is a shift-invariant extreme point of $\Gibbs(\Phi^\A)$ and $\mu_\B$ is any element of $\Gibbs^\Gamma(\Phi^\B)$, then in fact their product joining is the only joining which is Gibbs for the sum interaction. In particular, for any $\Sigma$ the product joining is the joining with maximal $\h_\Sigma$.

Corollary \ref{cor:nonzero} shows that, except in degenerate cases, Gibbs measures have nonzero sofic entropy over any deterministic sofic approximation.\\

Our final main result is Theorem \ref{thm:main3}, which asserts that, for free-boundary Ising models at low temperatures, the self-joining with maximal $\f$-invariant is neither the product nor the diagonal joining. Non-maximality of the diagonal joining actually follows in much greater generality from Theorem \ref{thm:gibbsmaximal}, since the diagonal joining is Gibbs only in degenerate cases. The product joining is always Gibbs for the sum interaction. But for temperatures low enough that the $\f$-invariant is negative, the product joining cannot be maximal because it has smaller $\f$-invariant than the diagonal.

Theorem \ref{thm:main3} actually extends non-maximality of the product to slightly higher temperatures.
To do this, we show that if the product joining of $\mu$ has optimal $\f$-invariant, then a typical random homomorphism supports good models for $\mu$. We can rule out this possibility for free-boundary Ising models at low temperatures using \cite{dembo2017}.

It remains open whether non-maximality of the product holds all the way up to the reconstruction threshold, at and above which the product joining is maximal by Corollary \ref{cor:prodmax}. A similar type of result in the recent paper \cite{coja-oghlan2020} suggests that it may. %They use a result of \cite{coja-oghlan2020a} which gives a variational principle for the free energy (i.e. normalized limit of the expected log partition functions) of the Ising model over a sequence of regular stochastic block models. This variational principle could be compared to \cite[Theorem C]{shriver2020a}, but both the notion of stochastic block model and the resulting optimization problem are quite different. \cite{coja-oghlan2020} shows that the free energy over this stochastic block model is strictly greater than the free energy over a random regular graph for all temperatures below the reconstruction threshold. The natural analogue of this result in the present setting would show that the product joining of the local weak$^*$ limit of the finitary Gibbs measures is not optimal in this range.

\subsection{Overview}
Section \ref{sec:definitions} contains setup of some of the basic objects of study, including the Glauber dynamics and Gibbs measures. In Section \ref{sec:metastability} we prove Theorem \ref{thm:Gibbsmetastability}, our main metastability result. In Section \ref{sec:application} we give an application of this theorem, characterizing which joinings of two Gibbs states have maximal sofic entropy over a random sofic approximation. Finally, in Section \ref{sec:nontriv} we show that, below a certain (nontrivial) temperature, the product self-joining of a free-boundary Ising state does not have maximal $\f$-invariant.

\subsection{Acknowledgements}
Thanks to Tim Austin for the suggestion of this project and for many helpful conversations and comments on earlier drafts. Thanks also to Lewis Bowen and Brandon Seward for helpful conversations.

This material is based upon work supported by the National Science Foundation under Grant No.~DMS-1855694.

\section{Definitions}
\label{sec:definitions}

For $\gamma \in \Gamma$, let $\abs{\gamma}$ denote the graph distance between $\gamma$ and the identity $e \in \Gamma$.

Give $\A^\Gamma$ the metric
	\[ d(\mathbf{x},\mathbf{y}) = \sum_{\gamma \in \Gamma} (3r)^{-\abs{\gamma}} \1_{\mathbf{x}(\gamma) \ne \mathbf{y}(\gamma)} ; \]
the factor 3 is chosen to ensure convergence. Note that $\diam \A^\Gamma \leq 3$. This metric induces the product topology (with $\A$ having the discrete topology).

Let $\bar{d}$ denote the corresponding transportation metric on $\Prob(\A^\Gamma)$ (the set of Borel probability measures); specifically, with $\Lip_1(\A^\Gamma)$ denoting the set of 1-Lipschitz real-valued functions, we define
	\[ \bar{d}(\mu, \nu) = \sup\left\{ \abs*{ \mu f - \nu f} \st f \in \Lip_1(\A^\Gamma) \right\} . \]
Here $\mu f$ denotes the integral of $f$ with respect to $\mu$.
Note that $d$ generates the product topology (which is compact), and $\bar{d}$ generates the weak topology induced by the pairing with continuous functions (which is also compact). \\

For any set $V$ and any $\mb{x} \in \A^V$, $v \in V$, $\ta \in \A$ we let $\mb{x}^{v \to \ta} \in \A^V$ be given by
	\[ \mb{x}^{v \to \ta}(w) = \left\{
						\begin{array}{ll}
							\mb{x}(w), 	& w \ne v \\
							\ta,			& w = v .
						\end{array} \right. \]
Below, an element of $\A^V$ will be referred to as a \emph{microstate} and an element of $\Prob(\A^V)$ as a \emph{state}.

\subsection{Interaction}
Let $V$ be an at most countable set and fix $\sigma \in \Hom(\Gamma, \Sym(V))$. We will apply this in two cases: when $V$ is finite, and when $V = \Gamma$ and $\sigma$ is the action of $\Gamma$ on itself by left multiplication. Below, we will distinguish between these cases by giving notation superscripts of $\sigma$ or $\Gamma$ respectively (e.g. $\Omega^\sigma$ versus $\Omega^\Gamma$).

A nearest-neighbor interaction with alphabet $\A$ is a pair $\Phi = (J,h)$ where $J \colon \A^2 \to \RR$ is symmetric and $h \colon \A \to \RR$. and let $S = \{ s_1, \ldots, s_r, s_1^{-1}, \ldots, s_r^{-1}\}$.
For $v \in V$, let $\Phi_v \colon \A^V \to \RR$ be given by
	\[ \Phi_v(\mb{x}) = h(\mb{x}(v)) + \sum_{s \in S} J(\mb{x}(v), \, \mb{x}(\sigma^s v)) . \]
	
If $V$ is finite then we can define the internal energy $U \colon \A^V \to \RR$ by
	\[ U(\mb{x}) = \sum_{v \in V} h (\mb{x}(v)) + \sum_{v \in V} \sum_{i \in [r]} J(\mb{x}(v),\, \mb{x}(\sigma^{s_i} v)) . \]
This can also be written
	\[ U(\mb{x}) = \sum_{v \in V} U_v(\mb{x}) \]
where
	\[ U_v(\mb{x}) = h (\mb{x}(v)) + \frac{1}{2} \sum_{s \in S} J(\mb{x}(v),\, \mb{x}(\sigma^{s} v)) . \]

An Ising model with no external field has $\A = \{-1, 1\}$, $J(\mb{x}) = \beta \ta \tb$, and $h \equiv 0$ for some $\beta \geq 0$ (the inverse temperature). The Bernoulli shift with base measure $p \in \Prob(\A)$ also fits into this framework by taking $J \equiv 0$ and $h(\ta) = -\log p(\{\ta\})$.

\subsection{Glauber dynamics}

For $\ta \in \A$ let
	\[ c_v(\mb{x}, \ta) = Z_v(\mb{x})^{-1} \exp \!\left\{ - \Phi_v(\mb{x}^{v \to \ta}) \right\} , \]
where $Z_v(\mb{x})$ is the normalizing factor which makes $c_v(\mb{x}, \cdot)$ a probability measure on $\A$. We can think of $c_v(\mb{x}, \cdot)$ as the transition rates for the spin at $v$ conditioned on the current state of the system being $\mb{x}$. Note that this only depends on the coordinates of $\mb{x}$ at vertices adjacent to $v$.

The Glauber dynamics is the continuous-time Markov process with state space $\A^V$ and generator $\Omega$ given by
	\[ \Omega f (\mb{x}) = \sum_{v \in V} \sum_{\ta \in \A} c_v(\mb{x}, \ta) [f(\mb{x}^{v \to \ta}) - f(\mb{x})] . \]
If $V$ is finite then this gives a well-defined linear operator on $C(\A^V)$. Otherwise we need to first define $\Omega$ on a `core' of `smooth' functions for which the sum converges, then take the closure of $\Omega$; see \cite{liggett2005} for details. The generator induces a Markov semigroup denoted $\{ S(t) : t \geq 0\}$.

Given $\mb{x} \in \A^V$, random or deterministic, we let $\mb{x}_t$ denote the $\A^V$-valued random variable which is the evolution of $\mb{x}$ to time $t$.

For any continuous function $f \colon \A^V \to \RR$ we interpret $S(t) f(\mb{x})$ as the expected value of $f(\mb{x}_t)$.

The semigroup also acts on probability measures, but on the right: $\mu S(t)$ is interpreted as the evolution of $\mu \in \Prob(\A^V)$ to time $t$. We will also often write $\mu_t \coloneqq \mu S(t)$; the relevant semigroup will typically be clear from context. The right action convention is appropriate because $[\mu S(t)] f = \mu [S(t) f]$, where $\mu f$ denotes the integral of $f$.\\

There is an approximate equivariance between the Glauber semigroups and the empirical distribution:

\begin{theorem}[{\cite{shriver2020a}}]
\label{thm:Sequivariance}
There is a constant $M > 0$ such that for any $\mb{x} \in \A^V$, $\sigma \in \Hom(\Gamma, \Sym(V))$, and $t\geq 0$
	\[ \bar{d}\left(S^\sigma(t) P_{\mb{x}}^\sigma,\ P_{\mb{x}}^\sigma S^\Gamma(t) \right) \leq \Delta^\sigma \cdot t e^{Mt} . \]
\end{theorem}
This theorem says that the expected empirical distribution after running the finitary dynamics for time $t$ is close to the (deterministic) result of evolving the original empirical distribution for time $t$, as long as $\sigma$ locally looks like $\Gamma$.

\subsection{Gibbs measures}

If $V$ is finite, the Gibbs measure $\xi_V \in \Prob(\A^V)$ is defined by
	\[ \xi_V \{\mb{x}\} = Z_V^{-1} \exp \{ - U(\mb{x}) \}  \]
where $Z_V$ is the normalizing constant.

On the infinite graph $\Gamma$ we must use a different approach, since the sum defining the total energy will not converge. We use a natural generalization of \cite[Definition IV.1.5]{liggett2005}; see also \cite{georgii2011} for a much more general treatment of infinite-volume Gibbs measures.

Let $\tail_\gamma$ denote the $\sigma$-algebra generated by all vertices except for $\gamma$. We call $\mu \in \Prob(\A^\Gamma)$ a Gibbs measure if for each $\gamma \in \Gamma$ and $\ta \in \A$, the function $\mb{y} \mapsto c_\gamma(\mb{y}, \ta)$ is a version of the conditional expectation $\mu(\{\mb{x} \st \mb{x}(\gamma) = \ta\} \mid \tail_v)(\mb{y})$. This means that for every integrable $f \colon \A^\Gamma \to \RR$ and $\gamma \in \Gamma$ we have
	\[ \int \sum_{\ta \in \A} c_\gamma( \mb{x}, \ta) f(\mb{x}^{\gamma \to \ta}) \mu(d\mb{x}) = \int f(\mb{x})\, \mu(d\mb{x}) . \]
We may also describe this relation by saying that $\mu$ is invariant under re-randomizing the spin at $\gamma$ using the kernel $c_\gamma$.

We will denote the set of all Gibbs measures for the interaction $\Phi$ by $\Gibbs(\Phi)$, or just $\Gibbs$ if the specific $\Phi$ is clear from context or irrelevant. The shift-invariant Gibbs measures will be denoted by $\Gibbs^\Gamma(\Phi)$ or $\Gibbs^\Gamma$.

The fact that $\Gibbs^\Gamma$ is a face of the simplex $\Prob^\Gamma(\A^\Gamma)$ will be important:
\begin{lemma}
\label{lem:Gibbsface}
	Let $\theta \in \Prob(\Prob^\Gamma(\A^\Gamma))$ and suppose $\int \mu\, \theta(d\mu) \in \Gibbs^\Gamma$. Then $\theta(\Gibbs^\Gamma) = 1$.
\end{lemma}
This is stated in the case $\Gamma = \ZZ^r$ in Georgii's book \cite[Theorem 14.15(c)]{georgii2011}. The proof works just as well in our generality, and goes as follows: It suffices to show that if $\mu, \nu \in \Prob(\A^\Gamma)$ are shift-invariant, $\mu \in \Gibbs^\Gamma$, and $\nu$ is absolutely continuous to $\mu$ then $\nu$ is also Gibbs. Under these assumptions, since $\nu \ll \mu$ we can write $\nu = f \mu$ for some measurable $f$. But since $\nu,\mu$ are shift-invariant, $f$ must be $\mu$-a.s.~equal to a shift-invariant function. Since $\mu$ is shift-invariant, the $\sigma$-algebra of shift-invariant measurable subsets of $\A^\Gamma$ is contained in the tail $\sigma$-algebra up to $\mu$-null sets. Therefore $f$ is $\mu$-a.s.~equal to a tail-measurable function. From this we can conclude that $\nu$ is Gibbs.

\subsection{Good models for measures on $\A^\Gamma$}

Let $V$ be a finite set and let $\sigma \in \Hom(\Gamma, \Sym(V))$. A labeling $\mb{x} \in \A^V$ is said to be a \emph{good model} for $\mu \in \Prob(\A^\Gamma)$ over $\sigma$ if the empirical distribution $P_{\mb{x}}^\sigma$ is close to $\mu$ in the weak topology. More precisely, we can say $\mb{x}$ is $\calO$-good if $P_{\mb{x}}^\sigma \in \calO$ for some weak-open neighborhood $\calO \ni \mu$. The set of such $\mb{x}$ is denoted $\Omega(\sigma, \calO)$. An interpretation of this relationship is that average local quantities of the finite system are consistent with $\mu$.

We define the empirical distribution of a state $\zeta \in \Prob(\A^V)$ by
	\[ P_\zeta^\sigma \coloneqq \zeta P_{\mb{x}}^\sigma = \int P_{\mb{x}}^\sigma\, \zeta(d\mb{x}) \in \Prob^\Gamma(\A^\Gamma) \]
and say that $\zeta$ is $\calO$-consistent with $\mu$ (for some neighborhood $\calO\ni \mu$) over $\sigma$ if $P_\zeta^\sigma \in \calO$. We can still interpret this in terms of averages of local quantities: now the average also involves a random microstate $\mb{x}$ with law $\zeta$. We denote the set of such states by $\bOmega(\sigma, \calO)$. This way of lifting a finitary state is used in \cite{alpeev2016}; it also essentially appears in the notion of ``local convergence on average'' introduced in \cite[Definition 2.3]{montanari2012}.

\section{Metastability of near-Gibbs-ness}
\label{sec:metastability}

The consistency of a state is stable under Glauber dynamics in the following sense:

\begin{prop}[\cite{shriver2020a}]
\label{prop:stability}
	Suppose $\sigma \in \Hom(\Gamma, \Sym(V))$, $\zeta \in \Prob(\A^V)$, and $\mu \in \Prob(\A^\Gamma)$. Let $\zeta_t, \mu_t$ denote their evolutions under Glauber dynamics on $\sigma, \Gamma$ respectively. Then for any $t \geq 0$
		\[ \bar{d} \big( P_{\zeta_t}^\sigma, \mu_t \big) \leq \big[ \Delta^\sigma t + \bar{d} \big( P_\zeta^\sigma,\, \mu \big) \big]\exp(Mt) \]
	for some $M > 0$ which depends only on the interaction and $\Gamma$.
\end{prop}

If we apply this with $\zeta = \delta_\mb{x}$ and $\mu \in \Gibbs^\Gamma$ we have
\begin{equation}
\label{eqn:expectationbound}
	\bar{d} \big( S^\sigma(t) P_\mb{x}^\sigma,\ \mu \big) \leq \big( \bar{d}\big( P_\mb{x}^\sigma ,\ \mu \big) + \Delta^\sigma t \big)e^{Mt} .
\end{equation}
In particular, if $\mathbf{x}$ is a good model over $\sigma$ for a Gibbs measure $\mu$, then the \emph{expected} empirical distribution of $\mb{x}_t$ stays close to $\mu$ for a long time. The first main theorem of the present paper is that, in fact, the empirical distribution itself stays close to $\mu$ for a long time with high probability:
\begin{mainthm}
\label{thm:Gibbsmetastability}
	Let $\mu \in \Gibbs^\Gamma$. Given any neighborhood $\calU_1$ of $\mu$ and $t,\varepsilon>0$, there exists a neighborhood $\calU_0$ of $\mu$ and $\delta>0$ such that if $\mb{x}_0 \in \Omega(\sigma,\calU_0)$ and $\Delta^\sigma < \delta$ then with probability at least $1-\varepsilon$ we have $\mb{x}_s \in \Omega(\sigma, \calU_1)$ for all $s \in [0,t]$.
\end{mainthm}
	
The remainder of this section is devoted to the proof of this theorem. First we use Lemma \ref{lem:Gibbsface} to show that Equation \ref{eqn:expectationbound} implies $P_{\mb{x}_t}^\sigma$ must stay close to $\Gibbs^\Gamma$ for a long time with high probability. We then control the `lateral motion,' showing that as long as $P_{\mb{x}_t}^\sigma$ stays close to $\Gibbs^\Gamma$ it does not move much at all.

\subsection{Concentration from Convexity}

Let $I$ denote the weak*-continuous map
\begin{align*}
	I \colon \Prob(\Prob^\Gamma(\A^\Gamma)) &\to \Prob^\Gamma(\A^\Gamma) \\
	\xi &\mapsto \int \nu\, \xi(d\nu) . 
\end{align*}

Lemma \ref{lem:Gibbsface} stated that $\theta(\Gibbs^\Gamma) = 1$ whenever $I(\theta) \in \Gibbs^\Gamma$. The following result is an approximate version of this: if $I(\theta)$ is close to $\Gibbs^\Gamma$, then most of the mass of $\theta$ must be close to $\Gibbs^\Gamma$.

\begin{prop}
	Given any weak* neighborhood $\calW$ of $\Gibbs^\Gamma$ and $\varepsilon > 0$, there exists a weak* neighborhood $\calU$ of $\Gibbs^\Gamma$ such that if $I(\xi) \in \calU$ then $\xi(\calW) > 1 - \varepsilon$.
\end{prop}

\begin{proof}
	By the portmanteau theorem, the set $\calE = \{ \xi \st \xi(\calW) > 1 - \varepsilon \}$ is weak*-open, and it clearly contains the set $\Prob(\Gibbs^\Gamma)$ of probability measures supported on $\Gibbs^\Gamma$. We complete the proof by contradiction: suppose that for each neighborhood $\calU$ of $\Gibbs^\Gamma$ the intersection $(I^{-1} \calU) \cap \calE^c$ is nonempty.
	
	For each $n \in \NN$, let $\calU_n$ be the set of measures within $\bar{d}$-distance $1/n$ of $\Gibbs^\Gamma$. By assumption, we can pick a sequence $\xi_n \in (I^{-1} \calU_n) \cap \calE^c$. Now $\Prob(\Prob^\Gamma(\A^\Gamma))$ is compact, so $\xi_n$ has some convergent subsequence $\xi_{n_j}$. Note the limit of this sequence must still be in the closed set $\calE^c$. By definition of the sets $\calU_n$ and continuity of $I$, the limit must also be in $I^{-1}(\Gibbs^\Gamma)$.
	
	But since $\Gibbs^\Gamma$ is a face of $\Prob^\Gamma(\A^\Gamma)$ (Lemma \ref{lem:Gibbsface}), in fact $I^{-1}(\Gibbs^\Gamma) = \Prob(\Gibbs^\Gamma) \subseteq \calE$. This is a contradiction, so there must exist some neighborhood $\calU$ of $\Gibbs^\Gamma$ with $I^{-1}\calU \subseteq \calE$.
\end{proof}

\begin{prop}
\label{prop:stayneargibbs}
	Let $\calW$ be a weak* neighborhood of $\Gibbs^\Gamma$. Let $\varepsilon, t > 0$. Then there exists a weak* neighborhood $\calU$ of $\Gibbs^\Gamma$ and $\delta>0$ such that if $P_{\mb{x}_0}^\sigma \in \calU$ and $\Delta^\sigma < \delta$, then $P_{\mb{x}_t}^\sigma \in \calW$ with probability at least $1-\varepsilon$.
\end{prop}

\begin{proof}
	The previous proposition guarantees the existence of a neighborhood $\calV$ of $\Gibbs^\Gamma$ such that if $S^\sigma(t) P_{\mb{x}_0}^\sigma \in \calV$ then $P_{\mb{x}_t}^\sigma \in \calW$ with probability at least $1-\varepsilon$.
	
	Since $\Gibbs^\Gamma$ is compact, we can pick $\eta>0$ such that $\bigcup_{\mu \in \Gibbs^\Gamma} \ball[\bar{d}]{\mu}{2\eta} \subset \calV$. Let $\calU = \bigcup_{\mu \in \Gibbs^\Gamma} \ball[\bar{d}]{\mu}{\eta e^{-Mt}}$ and let $\delta = \eta e^{-Mt}/t$. Then by (\ref{eqn:expectationbound}) whenever $P_{\mb{x}_0}^\sigma \in \calU$ and $\Delta^\sigma < \delta$ we have $S^\sigma(t) P_\mb{x}^\sigma \in \calV$.
\end{proof}

%%%%%%%%%%%%%%%%%%%%%%%%%
%%%%%%%%%%%%%%%%%%%%%%%%%
%%%%%%%%%%%%%%%%%%%%%%%%%

\subsection{Controlling lateral motion}

Having shown that Glauber dynamics tends to stay within the set of good models for near-Gibbs measures, we now show that it tends to move slowly within this region.

Given $\mb{x}_0 \in \A^V$, $\sigma \in \Hom(\Gamma, \Sym(V))$, $g \colon \A^\Gamma \to \RR$, and $\tau > 0$, we define a martingale $(M^{g, \tau}_k)_{k = 0}^\infty$ by
	\[ M^{g, \tau}_k = P_{{\mb{x}}_{k\tau}}^\sigma g - P_{\mb{x}_0}^\sigma g - \sum_{s=0}^{k-1} L_\tau P_{{\mb{x}}_{s \tau}}^\sigma g \]
where
	\[ L_\tau \coloneqq S^\sigma(\tau) - I . \]
	
We first show that the terms in the sum stay small as long as $P_{{\mb{x}}_{s \tau}}^\sigma$ stays close to $\Gibbs^\Gamma$ (which we know is likely to happen as long as $P_{{\mb{x}}_0}^\sigma$ is close enough to $\Gibbs^\Gamma$), then we show that the martingale itself likely stays small by bounding the variance. This will imply that $P_{{\mb{x}}_t}^\sigma g$ tends to stay near its initial value.
	
\subsubsection{Bounding deviation from martingale}

It is straightforward from the definitions to show that $\mu \in \Gibbs^\Gamma$ then $\mu \Omega^\Gamma = 0$. We now show that if $\mu$ is near $\Gibbs^\Gamma$ then $\mu \Omega^\Gamma$ is near 0.

Let $\norm{f}_{\BL} = \max\{ \abs{f}_{\Lip}, \norm{f}_\infty\}$ denote the bounded Lipschitz norm of a real-valued function on $\A^\Gamma$. Under this norm, the set $\{ f \st \norm{f}_{\BL} < \infty\}$ is a Banach space which we call $\BL$. Every $\mu \in \Prob(\A^\Gamma)$ induces a continuous linear functional $I_\mu$ on $\BL$ defined by
	\[ I_\mu f = \int f\, d\mu . \]
If we endow the continuous dual $\BL^*$ with the standard dual (operator) norm, it is easy to see that
	\[ \bar{d}(\mu, \nu) = \norm{I_\mu - I_\nu}_{\BL^*}. \]
	
Since $\Omega^\Gamma g$ is a continuous function whenever $g \in \BL$, for any $\mu \in \Prob(\A^\Gamma)$ we can define $\mu \Omega^\Gamma \in \BL^*$ by
	\[ \big(\mu \Omega^\Gamma \big) g \coloneqq \int \Omega^\Gamma g\, d\mu \quad \forall g \in \BL . \]

\begin{lemma}
\label{lem:generatorcontinuity}
	The map
	\begin{align*}
		\Prob(\A^\Gamma) &\to \BL^* \\
		\mu &\mapsto \mu \Omega^\Gamma
	\end{align*}
	is continuous.
\end{lemma}
\begin{proof}
We first show that the family $F = \{ \Omega^\Gamma f \st \norm{f}_{\BL} \leq 1 \}$ is uniformly bounded and equicontinuous. Uniform boundedness is fairly straightforward. We now establish equicontinuity: Suppose $\mb{x}, \mb{y} \in \A^\Gamma$ are such that $d({\mb{x}},\mb{y}) < (3r)^{-k}$; then $\mb{x}$ and $\mb{y}$ agree on $\ball{e}{k}$ so for all $\gamma \in \ball{e}{k-1}$ and $\ta \in \A$ we have
	\[ c_\gamma({\mb{x}}, \ta) = c_\gamma(\mb{y}, \ta) . \]
So for such $\gamma$, if $\norm{f}_{\BL} \leq 1$ we have
\begin{align*}
	&\hspace{-2cm} \abs[\big]{c_{\gamma}({\mb{x}}, \ta) \left[ f({\mb{x}}^{\gamma \to \ta}) - f({\mb{x}}) \right] - c_\gamma (\mb{y}, \ta) \left[ f(\mb{y}^{\gamma \to \ta}) - f(\mb{y}) \right]} \\
		&\leq c_\gamma (\mb{x}, \ta) \big[ \abs{f({\mb{x}}) - f(\mb{y})} + \abs{f({\mb{x}}^{\gamma \to \ta}) - f(\mb{y}^{\gamma \to \ta})} \big] \\
		&\leq c_\gamma ({\mb{x}}, \ta) \big[2 \cdot (2r)^{-k} \big] .
\end{align*}
For $\gamma \not\in \ball{e}{k-1}$ we have $d(\mb{x}^{\gamma \to \ta}, \mb{x}) = (3r)^{-\abs{\gamma}}$, so
	\[ \abs[\big]{c_{\gamma}({\mb{x}}, \ta) \left[ f({\mb{x}}^{\gamma \to \ta}) - f({\mb{x}}) \right] - c_\gamma (\mb{y}, \ta) \left[ f(\mb{y}^{\gamma \to \ta}) - f(\mb{y}) \right]} \leq \abs[\big]{c_{\gamma}({\mb{x}}, \ta) - c_\gamma (\mb{y}, \ta)} (3r)^{-\abs{\gamma}} . \]
Hence
\begin{align*}
	\abs{\Omega^\Gamma f({\mb{x}}) - \Omega^\Gamma f(\mb{y})}
		&\leq \sum_{\gamma \in \ball{e}{k-1}} \sum_{\ta \in \A} 2 c_\gamma({\mb{x}}, \ta) \cdot (3r)^{-k} \\
			&\qquad + \sum_{\gamma \not\in \ball{e}{k-1}} \sum_{\ta \in \A} ( c_\gamma({\mb{x}}, \ta) + c_\gamma(\mb{y}, \ta)) (3r)^{-\abs{\gamma}}  \\
		&\leq 2 \abs{\ball{e}{k-1}} (3r)^{-k} + 2 \sum_{s = k}^\infty r^s (3r)^{-s} \\
		&= 2 \abs{\ball{e}{k-1}} (3r)^{-k} + 3 \cdot 3^{-k} \\
		&= o_{k \to \infty}(1).
\end{align*}
Since this bound is uniform over $f \in F$, the family $F$ is equicontinuous.

Suppose $(\mu_n)_{n=1}^\infty$ is a sequence of probability measures with weak* limit $\nu$. For any $\varepsilon>0$, by Arzel\`a-Ascoli we can pick a finite collection $\Omega^\Gamma f_1, \ldots, \Omega^\Gamma f_k \in F$ which is uniformly $\varepsilon$-dense in $F$. Hence
\begin{align*}
	\limsup_{n \to \infty} \norm{\mu_n\Omega^\Gamma - \nu\Omega^\Gamma}_{\BL^*} &= \limsup_{n \to \infty} \sup \left\{ \abs*{\int \Omega^\Gamma f\, d\mu - \int \Omega^\Gamma f\, d\nu} \st \norm{f}_{BL} \leq 1 \right\} \\
	&\leq \limsup_{n \to \infty} \left[ \max \left\{ \abs*{\int \Omega^\Gamma f_i\, d\mu_n - \int \Omega^\Gamma f_i\, d\nu} \st 1 \leq i \leq k \right\} + 2 \varepsilon \right] \\
	&= 2\varepsilon.
\end{align*}
Since $\varepsilon$ is arbitrary, this shows that $\mu_n \Omega^\Gamma$ converges to $\nu \Omega^\Gamma$.
\end{proof}

%We base this on the following estimate. Let
%	\[ \abs{f}_{\alpha} = \sup_{{\mb{x}} \ne y} \frac{\abs{f({\mb{x}}) -f(y)}}{(d({\mb{x}},y))^{\alpha}} \]
%denote the $\alpha$-H\"older seminorm. Let $\norm{f}_{BH^\alpha} = \max\{\abs{f}_\alpha, \norm{f}_\infty\}$.
%\begin{prop}
%	There is some constant $C$ such that if $\norm{f}_{BH^4} \leq 1$ then $\norm{Q^\Gamma f}_{BL} \leq C$.
%\end{prop}
%The previous proposition follows because the distance
%	\[ d^4(\mu, \nu) = \sup \left\{ \abs*{\int f\, d\mu - \int f\, d\nu} \st \norm{f}_{BH^4} \leq 1 \right\} \]
%generates the weak* topology on $\Prob(\A^\Gamma)$.

\begin{prop}
\label{prop:generatorbound}
	For any $\mb{x} \in \A^V$, $\tau > 0$, and $g$ with $\abs{g}_{\Lip} \leq 1$
		\[ \abs{L_\tau(P_{\mb{x}}^\sigma g)} \leq \tau\left[\Delta^\sigma e^{M\tau} + \norm*{\frac{S^\Gamma(\tau) g - g}{\tau} - \Omega^\Gamma g}_{\infty} + \norm{P_{\mb{x}}^\sigma \Omega^\Gamma}_{\BL^*} \right] . \]
\end{prop}
\begin{proof}
	For any $g$,
	\begin{align*}
		\abs{L_\tau(P_{\mb{x}}^\sigma g)} &= \abs{S^\sigma(\tau)P_{\mb{x}}^\sigma g - P_{\mb{x}}^\sigma g} \\
		&\leq \abs{S^\sigma(\tau)P_{\mb{x}}^\sigma g - P_{\mb{x}}^\sigma S^\Gamma (\tau) g} + \abs{P_{\mb{x}}^\sigma S^\Gamma(\tau) g - P_{\mb{x}}^\sigma g} .
	\end{align*}
	By Theorem \ref{thm:Sequivariance}, if $\abs{g}_{\Lip} \leq 1$ then the first term here is bounded by $\Delta^\sigma \tau e^{M\tau}$. For the second term, we have
	\begin{align*}
		\abs{P_{\mb{x}}^\sigma S^\Gamma(\tau) g - P_{\mb{x}}^\sigma g}
			&\leq \tau \left[ P_{\mb{x}}^\sigma \abs*{\frac{S^\Gamma(\tau) g - g}{\tau} - \Omega^\Gamma g} + \abs*{P_{\mb{x}}^\sigma \Omega^\Gamma g} \right] \\
			&\leq \tau \left[ \norm*{\frac{S^\Gamma(\tau) g - g}{\tau} - \Omega^\Gamma g}_{\infty} + \norm{P_{\mb{x}}^\sigma \Omega^\Gamma}_{\BL^*} \right]. \qedhere
	\end{align*}
\end{proof}

\subsubsection{Martingale concentration}

\begin{prop}
\label{prop:martingalebound}
	Fix $t > 0$ and $g$ with $\abs{g}_{\Lip} \leq 1$. Then for any $m \in \ZZ$ we have
		%\[  \PP\big( \abs{M^{g, t/m}_m} > \kappa \big) < \frac{9t}{\kappa^2} \left[\frac{1}{\abs{V}} + \frac{t}{m} \right] . \]
		\[ \EE \big[ (M^{g, t/m}_m)^2 \big] \leq 9 t \left[\frac{1}{\abs{V}} + \frac{t}{m} \right] . \]
\end{prop}

\begin{proof}
Let $\tau = t/m$. Let $\xi_1, \xi_2, \ldots$ denote the martingale increments given by
	\[ \xi_k = M^{g, \tau}_k - M^{g,\tau}_{k-1} = P_{{\mb{x}}_{k\tau}}^\sigma g - P_{{\mb{x}}_{(k-1)\tau}}^\sigma g - L_\tau P_{{\mb{x}}_{(k-1)\tau}}^\sigma g . \]
	
Let $K_k$ be the number of times a spin changes in the Glauber dynamics starting at $\mb{x}_0$ during the time interval $[(k-1)\tau, k\tau)$. We will use that $K_k$ is Poisson with mean $\tau \abs{V}$.

We need the following two lemmas:

\begin{lemma}
\label{lem:stepbound}
	If $\mb{x}, \mb{x}' \in \A^V$ differ at exactly one site $w \in V$, then
		\[ \abs*{ P_{\mb{x}}^\sigma g - P_{\mb{x}'}^\sigma g} \leq \frac{3}{\abs{V}}. \]
\end{lemma}
\begin{proof}
	Recall that we are assuming $\abs{g}_{\Lip} \leq 1$. Using the definitions of empirical distribution and the distance on $\A^V$,
	\begin{align*}
		\abs*{ P_{\mb{x}}^\sigma g - P_{\mb{x}'}^\sigma g} \leq \frac{3}{\abs{V}}
			&\leq \frac{1}{\abs{V}} \sum_{v \in V} \abs{ g(\Pi_v^\sigma \mb{x}) - g(\Pi_v^\sigma \mb{x}')} \\
			&\leq \frac{1}{\abs{V}} \sum_{v \in V} d( \Pi_v^\sigma \mb{x}, \Pi_v^\sigma \mb{x}) \\
			&= \frac{1}{\abs{V}} \sum_{v \in V} \sum_{\gamma \in \Gamma} (3r)^{-\abs{\gamma}} \1\{ \mb{x}(\sigma^\gamma v) \ne \mb{x}'(\sigma^\gamma v) \} .
	\end{align*}
	By assumption, $\mb{x}(\sigma^\gamma v) \ne \mb{x}'(\sigma^\gamma v)$ if and only if $\sigma^\gamma v = w$. Using this fact and changing the order of summation gives
		\[ \abs*{ P_{\mb{x}}^\sigma g - P_{\mb{x}'}^\sigma g} \leq \frac{3}{\abs{V}}
			\leq \frac{1}{\abs{V}} \sum_{\gamma \in \Gamma} (3r)^{-\abs{\gamma}} \sum_{v \in V} \1\{ \sigma^\gamma v = w \} . \]
	But $\sigma^\gamma$ is a permutation, so $\sum_{v \in V} \1\{ \sigma^\gamma v = w \} = 1$.
	The result now follows from the bound $\sum_{\gamma \in \Gamma} (3r)^{-\abs{\gamma}} \leq 3$.
\end{proof}

\begin{lemma}
\label{lem:momentbound}
	For any $k \in \NN$, 
		\[ \EE[ \xi_k^2 ] \leq 9 \abs{V}^{-2} \EE[K_k^2 ] .\]
\end{lemma}
\begin{proof}
	For each $k$ let $\calF_k$ be the $\sigma$-algebra generated by $(\mb{x}_0, \mb{x}_\tau, \ldots, \mb{x}_{k \tau})$.
	
	We first expand out $\xi_k^2$ using its definition, and then simplify the resulting expression using that $L_\tau P_{{\mb{x}}_{(k-1)\tau}}^\sigma g$ is $\calF_{k-1}$-measurable:
    \begin{align*}
    	\EE[ \xi_k^2 \mid \calF_{k-1}] &= \EE \left[ \left( P_{{\mb{x}}_{k\tau}}^\sigma g - P_{{\mb{x}}_{(k-1)\tau}}^\sigma g - L_\tau P_{{\mb{x}}_{(k-1)\tau}}^\sigma g \right)^2 \mid \calF_{k-1} \right] \\
	&= \EE \left[ \left( P_{{\mb{x}}_{k\tau}}^\sigma g - P_{{\mb{x}}_{(k-1)\tau}}^\sigma g \right)^2 \mid \calF_{k-1} \right]
		+ \EE \left[ \left( L_\tau P_{{\mb{x}}_{(k-1)\tau}}^\sigma g \right)^2 \mid \calF_{k-1} \right] \\
		&\qquad
			+ \EE \left[ 2 \left(P_{{\mb{x}}_{k\tau}}^\sigma g - P_{{\mb{x}}_{(k-1)\tau}}^\sigma g\right)\left(- L_\tau P_{{\mb{x}}_{(k-1)\tau}}^\sigma g \right) \mid \calF_{k-1} \right] \\
	&= \EE \left[ \left( P_{{\mb{x}}_{k\tau}}^\sigma g - P_{{\mb{x}}_{(k-1)\tau}}^\sigma g \right)^2 \mid \calF_{k-1} \right]
		+ \left( L_\tau P_{{\mb{x}}_{(k-1)\tau}}^\sigma g \right)^2 \\
		&\qquad
			- 2 L_\tau P_{{\mb{x}}_{(k-1)\tau}}^\sigma g \cdot \EE \left[ P_{{\mb{x}}_{k\tau}}^\sigma g - P_{{\mb{x}}_{(k-1)\tau}}^\sigma g \mid \calF_{k-1} \right] \\
	&= \EE \left[ \left( P_{{\mb{x}}_{k\tau}}^\sigma g - P_{{\mb{x}}_{(k-1)\tau}}^\sigma g \right)^2 \mid \calF_{k-1} \right]
		- \left( L_\tau P_{{\mb{x}}_{(k-1)\tau}}^\sigma g \right)^2 .
    \end{align*}
    Dropping the second term, we're left with
		\[ \EE[ \xi_k^2 \mid \calF_{k-1}] \leq \EE \left[ \left( P_{{\mb{x}}_{k\tau}}^\sigma g - P_{{\mb{x}}_{(k-1)\tau}}^\sigma g \right)^2 \mid \calF_{k-1} \right] . \]

	By the previous lemma, each of the $K_k$ spin flips moves $P_{\mb{x}}^\sigma g$ by at most $3/\abs{V}$, so we have
    		\[ \abs{P_{{\mb{x}}_{k\tau}}^\sigma g - P_{{\mb{x}}_{(k-1)\tau}}^\sigma g} \leq \tfrac{3}{\abs{V}} K_k . \]
	Putting this into the previously obtained bound and taking expectations gives the claimed result.
\end{proof}
	
Using Lemma \ref{lem:momentbound} and that $K_k \sim \mathrm{Pois}(\tau \abs{V})$, we see that
	\[ \EE[ \xi_k^2] \leq 9 \abs{V}^{-2} \EE[K_k^2] = \tfrac{9}{\abs{V}^2} [ \tau \abs{V} + (\tau \abs{V})^2] = 9 \tau[ \tfrac{1}{\abs{V}} + \tau] . \]
Therefore, since the martingale increments are uncorrelated,
	\[ \EE \big[ (M^{g, t/m}_m)^2 \big] = \sum_{k=1}^mm \EE[ \xi_k^2] \leq 9 t [\tfrac{1}{\abs{V}} + \tfrac{t}{m}] . \qedhere \]
\end{proof}

By similar methods we can prove the following lemma, which controls the empirical distribution at times between multiples of $t/m$:

\begin{lemma}
\label{lem:betweentimes}
	For any $t,\kappa>0$, $g$ with $\norm{g}_{BL} \leq 1$, $m \in \NN$, and $0 \leq j \leq m-1$,
		\[ \PP\big( \exists s \in [jt/m,(j+1)t/m] \text{ with } \abs{P_{\mb{x}_s}^\sigma g - P_{\mb{x}_{jt/m}}^\sigma g} > \kappa \big)
			\leq 9 \frac{t}{m} \left[ \frac{1}{\abs{V}} + \frac{t}{m}\right] . \] 
\end{lemma}
\begin{proof}
	By Lemma \ref{lem:stepbound} the probability is bounded above by
		\[ \PP( K_j > \abs{V} \kappa / 3). \]
	The result follows from applying Chebyshev's inequality, using that $K_j$ has law $\mathrm{Pois}(t \abs{V} /m)$.
\end{proof}

\subsection{Proof of Theorem \ref{thm:Gibbsmetastability}}

Fix $\kappa > 0$ such that $\ball{\mu}{9\kappa} \subset \calU_1$. Let $F$ be a finite $\kappa$-dense (in uniform norm) subset of $\{ f \st \norm{f}_{BL} \leq 1\}$; we showed above that this set is compact in the uniform norm. Then for any $\mu, \nu \in \Prob(\A^\Gamma)$,
	\[ \bar{d}(\mu, \nu) \leq \sup \left\{ \abs*{\mu g - \nu g} \st g \in F \right\} + 2 \kappa. \]
	
For any given $t, \varepsilon > 0$, by Proposition \ref{prop:martingalebound} and Doob's maximal inequality we can pick $M \in \NN$ such that for any $\mathbf{x}_0 \in \A^V$
	\[ \PP\left( \max_{g \in F}\max_{0 \leq j \leq m} \abs{M^{g, t/m}_j} \leq \kappa \right) \geq 1- \varepsilon \tag{*} \]
whenever $\abs{V},m \geq M$ (recall that the martingale has an implicit dependence on a choice of initial microstate $\mb{x}_0 \in \A^V$). By Lemma \ref{lem:betweentimes} we can make $M$ larger if necessary to also ensure that for each $0 \leq j \leq m$ we have
	\[ \PP\big(\max_{g \in F} \max_{s \in [0,t]} \abs{P_{\mb{x}_s}^\sigma g - P_{\mb{x}_{\floor{sm/t}t/m}}^\sigma g} \leq \kappa \big)
			\geq 1- \varepsilon . \tag{**}\]

Assume that $m$ is also large enough that
	\[ \norm*{\frac{S^\Gamma(t/m) g - g}{t/m} - \Omega^\Gamma g}_{\infty} \leq \kappa/t \]
for every $g \in F$ and that $e^{Mt/m} \leq 2$. Assume also that $\Delta^\sigma \leq \kappa/t$ and let $\calW = \{ \nu \st \norm{\nu \Omega^\Gamma}_{\BL^*} < \kappa/t\}$; this is an open neighborhood of $\Gibbs^\Gamma$ by continuity of the map $\nu \mapsto \nu \Omega^\Gamma$ (Lemma~\ref{lem:generatorcontinuity}). Then, by Proposition \ref{prop:generatorbound}, $P_{\mb{x}}^\sigma \in \calW$ implies
	\[ \abs{L_{t/m}(P_{\mb{x}}^\sigma g)} \leq \tfrac{4\kappa}{m} . \]
We have also shown (Proposition~\ref{prop:stayneargibbs}) that there exist a weak neighborhood $\calU$ of $\Gibbs^\Gamma$ and $\delta>0$ such that if $P_{{\mb{x}}_0}^\sigma \in \calU$ and $\Delta^\sigma < \delta$, then for each $s \leq t$ we have $\PP( P_{{\mb{x}}_s}^\sigma \in \calW) \geq 1 - \varepsilon/m$. Therefore under these assumptions
	\[ \PP( P_{{\mb{x}}_{kt/m}}^\sigma \in \calW \ \forall 0 \leq k \leq m) \geq 1 - \varepsilon . \tag{***} \]

Suppose that $\mb{x}_0 \in \Omega(\sigma, \calU)$. Then the the probability that the events appearing in (*), (**), and (***) all occur is at least $1-3\varepsilon$. Assume they do all occur. Given $g \in F$ and $s \in [0,t]$, pick $j = \floor{sm/t}$. Then $0 \leq j \leq m$ so we have
\begin{align*}
	\abs{P_{{\mb{x}}_{s}}^\sigma g - \mu g}
		&\leq \abs{P_{{\mb{x}}_{jt/m}}^\sigma g - \mu g} + \kappa \\
		&\leq \abs{M_j^{g,t/m}} + \abs{P_{{\mb{x}}_0}^\sigma g - \mu g} + \sum_{k=0}^{j-1} \abs{L_{t/m}(P_{{\mb{x}}_{kt/m}}^\sigma g)} + \kappa \\
		&\leq \kappa + \bar{d}(P_{{\mb{x}}_0}^\sigma, \mu) + j \cdot \tfrac{4 \kappa}{m} + \kappa.
\end{align*}
So if also $\bar{d}(P_{{\mb{x}}_0}^\sigma, \mu) \leq \kappa$ then for any $s \in [0,t]$ we have
	\[ \sup \{ \abs{P_{\mb{x}_s} g - \mu g} \st g \in F \} \leq 7 \kappa \]
so
	\[ \bar{d}(P_{{\mb{x}}_s}^\sigma, \mu) \leq 9 \kappa \]
and hence $P_{{\mb{x}}_{s}}^\sigma \in \calU_1$.

In summary: let $\calU_0 = \calU \cap \ball{\mu}{\kappa}$. If $P_{{\mb{x}}_0}^\sigma \in \calU_0$, $\abs{V} \geq M$, and $\Delta^\sigma < \delta$, then with probability at least $1-3\varepsilon$ we have $P_{{\mb{x}}_{s}}^\sigma \in \calU_1$ for all $s \in [0,t]$. Since $\Delta^\sigma$ can only be small if $\abs{V}$ is large, we can remove the explicit requirement of a lower bound on $\abs{V}$ by making $\delta$ smaller if necessary.

%%%%%%%%%%%%%%%%%%%%
%%%%%%%%%%%%%%%%%%%%
%%%%%%%%%%%%%%%%%%%%
%%%%%%%%%%%%%%%%%%%%

%%%%%%%%%%%%%%%%%%%%
%%%%%%%%%%%%%%%%%%%%
%%%%%%%%%%%%%%%%%%%%
%%%%%%%%%%%%%%%%%%%%

\section{Maximal entropy joinings}
\label{sec:application}

We will call a sequence of random homomorphisms $\Sigma = (\sigma_n \in \Hom(\Gamma, \Sym(V_n)))_{n \in \NN}$ a \emph{random sofic approximation} to $\Gamma$ if for any $\delta>0$ there exists $c > 0$ such that
	\[ \PP( \Delta^{\sigma_n} > \delta) < n^{-cn} . \]
Examples include deterministic sofic approximations by homomorphisms, uniformly random homomorphisms, and stochastic block models \cite{shriver2020}. The assumption that the maps be true homomorphisms has been adopted for simplicity and with a particular application in mind, but is probably not necessary; see \cite{airey2019} for a more general definition.

Write the exponential growth rate for the expected number of good models for $\mu \in \Prob^\Gamma(\A^\Gamma)$ as
		\[ \h_\Sigma(\mu) = \inf_{\calO \ni \mu} \limsup_{n \to \infty} \frac{1}{\abs{V_n}} \log \EE \abs*{\Omega(\sigma_n, \calO)} . \]
If every term of $\Sigma$ is deterministic then this is the standard sofic entropy. If $\Gamma$ is a free group and each term of $\Sigma$ is uniform then this is the $\f$-invariant \cite{bowen2010}.

Given two measures $\mu_\A \in \Prob^\Gamma(\A^\Gamma)$ and $\mu_\B \in \Prob^\Gamma(\B^\Gamma)$, which joinings of the two maximize $\h_\Sigma$ for a fixed $\Sigma$? This question arises in \cite{shriver2020} and may be of more general interest.

The following theorem provides some information in the case where both systems are Gibbs measures for nearest-neighbor interactions. To state it we need one definition, which is a particular case of \cite[Example 7.18]{georgii2011}: given two nearest-neighbor interactions $\Phi^\A = (J^\A, h^\A), \Phi^\B = (J^\B,h^\B)$ with respective finite alphabets $\A,\B$, define their sum $\Phi^\A \oplus \Phi^\B$ to be the pair $(J^\A\oplus J^\B, h^\A \oplus h^\B)$, where
\begin{align*}
	[J^\A\oplus J^\B] \big( (\ta_1,\tb_1), (\ta_2,\tb_2) \big) &= J^\A(\ta_1,\ta_2) + J^\B(\tb_1, \tb_2) \\
	[h^\A \oplus h^\B] (\ta, \tb) &= h^\A(\ta) + h^\B(\tb).
\end{align*}
This is a nearest-neighbor interaction with alphabet $\A \times \B$.

We will use $c^\A_v(\mb{x}, \cdot) \in \Prob(\A)$ to refer to the transition rates for the Glauber dynamics of $\Phi^\A$, and $U^\A \colon \A^V \to \RR$ to refer to the energy, and similarly for $c^\B, U^\B$. Without superscripts, $c,U$ will refer to $\Phi^\A \oplus \Phi^\B$. Note that if $\mb{x} \in (\A \times \B)^V$ then $c_v(\mb{x}, (\ta,\tb)) = c^\A_v(\mb{x}_\A, \ta) c^\B_v(\mb{x}_\B, \tb)$ and $U(\mb{x}) = U^\A(\mb{x}_\A) + U^\B(\mb{x}_\B)$. In particular, the Glauber dynamics for $\Phi^\A \oplus \Phi^\B$ is a coupling of the Glauber dynamics of the summands.

If $\mu_\A \in \Gibbs^\Gamma(\Phi^\A)$ and $\mu_\B \in \Gibbs^\Gamma(\Phi^\B)$, then $\mu_\A \times \mu_\B \in \Gibbs^\Gamma(\Phi^\A \oplus \Phi^\B)$; in particular, there always exist joinings which are Gibbs for the sum interaction.

\begin{mainthm}
\label{thm:gibbsmaximal}
	Let $\lambda$ be a joining of two Gibbs measures $\mu_\A \in \Prob(\A^\Gamma), \mu_\B \in \Prob(\B^\Gamma)$ for nearest-neighbor interactions $\Phi^\A, \Phi^\B$ respectively. Let $\Sigma$ be a random sofic approximation to $\Gamma$, and assume that $\h_\Sigma$ is not identically $-\infty$ on $\join(\mu_\A, \mu_\B)$.
	
	If $\lambda$ maximizes $\h_\Sigma$ among all joinings of $\mu_\A, \mu_\B$, then $\lambda \in \Gibbs^\Gamma (\Phi^\A \oplus \Phi^\B)$.
\end{mainthm}
In particular, since $\h_\Sigma$ is upper semicontinuous, there is a Gibbs joining which has maximal $\h_\Sigma$ among all joinings.

By \cite[Equation (7.19)]{georgii2011}, we have
	\[ \ex \Gibbs(\Phi^\A \oplus \Phi^\B) = \{ \mu \times \nu \st \mu \in \ex \Gibbs(\Phi^\A),\ \nu \in \ex\Gibbs(\Phi^\B) \} . \]
Therefore the previous theorem implies that a maximal-entropy joining of two Gibbs measures must be a relative product over the tail $\sigma$-algebra.

Consider the following theorem:
\begin{theorem}[{\cite[Theorem B]{shriver2020a}}]
\label{thm:holleyext}
	Suppose $\mu \in \Prob^\Gamma(\A^\Gamma)$, and let $\mu_t$ denote its evolution under the Glauber dynamics for a nearest-neighbor potential $\Phi$. If $\Gamma$ has property PA\footnote{We will not need a definition here, but one can be found in \cite{shriver2020a}}, then $\mu_t$ converges weakly to $\Gibbs^\Gamma(\Phi)$ as $t \to \infty$.
\end{theorem}

One could prove our Theorem \ref{thm:gibbsmaximal} using this theorem roughly as follows: Let $\Phi = \Phi^\A \oplus \Phi^\B$. Starting with an arbitrary $\lambda \in \join(\mu_\A, \mu_\B)$, if we evolve under Glauber dynamics for $\Phi$ then eventually $\lambda_t$ will become as close as we like to $\Gibbs^\Gamma(\Phi)$, while staying in $\join(\mu_\A, \mu_\B)$ (since the marginals are invariant). If we evolve a collection of good models for $\lambda$ for the same amount of time, our metastability result (Theorem \ref{thm:Gibbsmetastability}) implies that they mostly stay good models for approximate joinings. It can also be shown that the evolved collection is almost as large as the initial one, and that most of the evolved states are good models for Gibbs states. From this we could conclude that there is a Gibbs state with at least as many good models as $\lambda$. %, and by being more careful we could show that if $\lambda$ is not Gibbs then the number of good models strictly increases.
%We can obtain a strict increase in the non-Gibbs case by considering the entropy of measures supported on good models rather than the size of a collection of good models.

However, it turns out to be easier to directly use the following proposition, which is the main technical result used to prove Theorem \ref{thm:holleyext}:

\begin{prop}[{\cite[Prop. 3.3 part 1]{shriver2020a}}]
\label{prop:strictdecrease}
Suppose $\mu \in \Prob^\Gamma(\A^\Gamma)$ is not Gibbs for some nearest-neighbor interaction $\Phi$. Then there exist $\delta,c,T>0$ and an open neighborhood $\calO \ni \mu$ such that such that for any $\sigma \in \Hom(\Gamma, \Sym(V))$ with $\Delta^\sigma < \delta$, if $\zeta_0 \in \mb\Omega(\sigma, \calO)$ then $A(\zeta_t) \leq A(\zeta_0) - c t \abs{V}$ for all $t \in [0,T]$ (here $\zeta_t$ refers to evolution of $\zeta$ under the Glauber dynamics for $\Phi$).
\end{prop}

Here is a brief summary of the proof of Theorem \ref{thm:gibbsmaximal}:

Suppose $\lambda$ is a joining of $\mu_\A,\mu_\B$ which is not Gibbs. Fix $n$ and $\calO \ni \lambda$. Let $p_0$ be the uniform distribution on $\Omega(\sigma_n, \calO) \subset (\A \times \B)^{V_n}$, and let $p_t$ denote its evolution under the Glauber dynamics for $\Phi$.

Since the marginals of $\lambda$ are Gibbs, and hence invariant under the Glauber dynamics, the average energy $p_t(U)$ is approximately constant over time. But we know that the free energy $A(p_t)$ is strictly decreasing since $\lambda$ is not Gibbs. This means that the Shannon entropy of $p_t$ must be strictly increasing (up to a small error). But the Shannon entropy of $p_0$ is $\log \abs{\Omega(\sigma_n, \calO)}$, and $p_t$ is mostly supported on good models for approximate joinings of $\mu_\A, \mu_\B$ (with the quality of the approximation getting better as $\Delta^{\sigma_n} \to 0$).

The evolved measure $p_t$ having strictly larger entropy means that its support, which is mostly good models for approximate joinings of $\mu_\A, \mu_\B$, must be strictly larger than the set of good models for the particular joining $\lambda$. This will imply that $\h_\Sigma(\lambda)$ is not maximal.

Note that we do not know whether $p_t$ stays mostly supported on good models for $\lambda_t$; we just know that its expected empirical distribution is near $\lambda_t$. So we cannot simply say that $\h_\Sigma(\lambda_t)$ is increasing.\\

The connection between entropy and the size of support is made using the following variant of Fano's inequality, standard versions of which can be found in \cite{cover2006}.
\begin{lemma}
\label{lem:fano}
	Let $\calF$ be a finite set and let $p \in \Prob(\calF)$. If $E \subset \calF$ satisfies $p(E) \geq 1-\varepsilon$ for some $\varepsilon > 0$ then
		\[ \log \abs{E} \geq \shent(p) - [\log 2 + \varepsilon \log \abs{\calF}] . \]
\end{lemma}
\begin{proof}
	Using the definition of Shannon entropy and splitting terms according to $E$ and its complement,
	\begin{align*}
		\shent(p)
			&= -\sum_{x \in E} p\{x\} \log p\{x\} -\sum_{x \not\in E} p\{x\} \log p\{x\} \\
			&= - \left[ p(E) \sum_{x \in E} \frac{p\{x\}}{p(E)} \log \frac{p\{x\}}{p(E)} + p(E) \log p(E) \right] \\
				&\qquad - \left[ (1-p(E)) \sum_{x \not\in E} \frac{p\{x\}}{1-p(E)} \log \frac{p\{x\}}{1-p(E)} + (1-p(E)) \log (1-p(E)) \right] .
	\end{align*}
	Let $p_E \in \Prob(E)$ denote the renormalized restriction of $p$ to $E$, and similarly define $p_{E^c} \in \Prob(E^c)$. Then the above can be written
	\begin{align*}
		\shent(p) 
			&= p(E) \shent(p_E) + (1-p(E)) \shent(p_{E^c}) + \shent( p(E), 1-p(E)) \\
			&\leq \log \abs{E} + \varepsilon \log \abs{\calF} + \log 2.
	\end{align*}
	Rearranging gives the claimed inequality.
\end{proof}

%\begin{lemma}
%\label{lem:energyconstant}
%	If $\theta \in \join^\eta(\mu_\A, \mu_\B)$ then
%		\[ \abs{\theta(U_e) - [\mu_\A U^\A_e + \mu_\B U^\B_e]} < \eta [ \abs{U^\A_e}_{\Lip} + \abs{U^\B_e}_{\Lip} ] , \]
%	where $\abs{U^\A_e}_{\Lip}, \abs{U^\B_e}_{\Lip}$, are Lipschitz constants of the functions $U^\A_e \colon \A^\Gamma \to \RR$, $U^\B_e \colon \B^\Gamma \to \RR$.
%\end{lemma}
%\begin{proof}
%	Since $\theta(U_e) = \pi_\A \theta(U^\A_e) + \pi_\B \theta (U^\B_e)$,
%	\begin{align*}
%		\abs{\theta(U) - [\mu_\A U^\A_e + \mu_\B U^\B_e]}
%			&\leq \abs{\pi_\A \theta(U^\A_e) - \mu_\A U^\A_e} +
%				\abs{\pi_\B \theta(U^\B_e) - \mu_\B U^\B_e} \\
%			&\leq \eta [ \abs{U^\A_e}_{\Lip} + \abs{U^\B_e}_{\Lip} ] . \qedhere
%	\end{align*}
%\end{proof}

The following proposition shows that the number of good models for any non-Gibbs joining is strictly smaller (by an exponential factor) than the number of good models for approximate joinings. 

\begin{prop}
\label{prop:modelsbound}
	Suppose $\lambda \in \join(\mu_\A, \mu_\B)$ is not in $\Gibbs(\Phi)$. There exist constants $C_1, C_2 > 0$ such that for any $\varepsilon,\eta>0$ there exist $\delta>0$ and $\calO \ni \lambda$ such that if $\sigma \in \Hom(\Gamma, \Sym(V))$ satisfies $\Delta^\sigma < \delta$ then
		\[ \abs{\Omega(\sigma, \join^\eta(\mu_\A, \mu_\B))} \geq \abs{\Omega(\sigma, \calO)} \cdot \tfrac{1}{2} \exp\!\big[ \abs{V} (C_1 - \varepsilon C_2) \big] . \]
\end{prop}

\begin{proof}
	Note that if $\Omega(\sigma, \calO)$ is empty then the inequality is trivially satisfied, so we will assume below that this is not the case.
	
	First, using that $\lambda$ is not Gibbs, pick $\delta, c, T > 0$ and $\calO \ni \lambda$ as appear in Proposition \ref{prop:strictdecrease}. Fix $t \in (0,T]$ arbitrarily.
	
	Note that for convenience we may assume $\varepsilon < \eta$. By Theorem \ref{thm:Gibbsmetastability}, by making $\delta, \calO$ smaller if necessary we can ensure that if $\mb{x}_0 \in \Omega(\sigma, \calO)$ and $\Delta^\sigma < \delta$ then $P_{\mb{x}_t}^\sigma \in \join^\varepsilon(\mu_\A, \mu_\B)$ with probability at least $1-\varepsilon$. Consequently, if we let $p_0 = \Unif(\Omega(\sigma, \calO))$ then $p_t(\Omega(\sigma, \join^\varepsilon(\mu_\A, \mu_\B))) > 1-\varepsilon$; note that since $\diam \A^\Gamma, \diam \B^\Gamma \leq 3$ this implies $P_{p_t}^\sigma \in \join^{4\varepsilon}(\mu_\A, \mu_\B)$. Also, for convenience we may shrink $\calO$ if necessary to ensure $\calO \subset \join^\varepsilon(\mu_\A, \mu_\B)$.
	
	We now show that the entropy of $p_t$ is increasing, up to a small error. By choice of $\delta, c, t, \calO$ we have
		\[ \fe(p_t) \leq \fe(p_0) - c \abs{V} t , \]
	or equivalently
		\[ \shent(p_t) \geq \shent(p_0) + c \abs{V} t - [p_0(U) - p_t(U)] . \]
	Since the empirical distributions of $p_0,p_t$ have approximately the same marginals, the difference in average energy is small. Specifically, since $P_{p_0}^\sigma \in \join^\varepsilon(\mu_\A, \mu_\B)$ and $P_{p_t}^\sigma \in \join^{4\varepsilon}(\mu_\A, \mu_\B)$
	\begin{align*}
		\abs{p_t(U) - p_0(U)}
			&= \abs{V} \cdot \abs{P_{p_t}^{\sigma} U_e - P_{p_0}^{\sigma} U_e} \\
			&\leq \abs{V} \cdot \big( \abs{P_{p_t}^{\sigma} U_e - [\mu_\A U^\A_e + \mu_\B U^\B_e]} + \abs{[\mu_\A U^\A_e + \mu_\B U^\B_e] - P_{p_0}^{\sigma} U_e} \big) \\
			&\leq \abs{V} \cdot \big( \abs{\pi_\A P_{p_t}^{\sigma} U^\A_e - \mu_\A U^\A_e} + \abs{\pi_\B P_{p_t}^{\sigma} U^\B_e - \mu_\B U^\B_e} \\
				&\quad + \abs{\pi_\A P_{p_0}^{\sigma} U^\A_e - \mu_\A U^\A_e} + \abs{\pi_\B P_{p_0}^{\sigma} U^\B_e - \mu_\A U^\B_e} \big) \\
			&\leq 5 \varepsilon \abs{V} ( \abs{U^\A_e}_{\Lip} + \abs{U^\B_e}_{\Lip} ) 
	\end{align*}
	so
		\[ \shent(p_t) \geq \shent(p_0) + c \abs{V} t - 5 \varepsilon \abs{V} ( \abs{U^\A_e}_{\Lip} + \abs{U^\B_e}_{\Lip} ) . \]	
%	We now apply the following version of Fano's inequality to bound the size of $\join^\varepsilon(\mu_\A, \mu_\B)$:
%	\begin{lemma}[{\cite[eqn. 2.131]{cover2006}}]
%	If $X,Y$ are random variables taking values in a finite set $\A$, then
%	\pushQED{\qed}
%		\[ \shent(X \mid Y) \leq 1 + \PP(X \ne Y) \log \abs{\A} . \qedhere \]
%	\popQED
%	\end{lemma}
%	Let $X$ be an $(\A \times \B)^{V}$-valued random variable with law $p_t$, and let $Y$ be an $(\A \times \B)^{V}$-valued random variable which takes values in $\Omega(\sigma, \join^\varepsilon(\mu_\A, \mu_\B))$ and equals $X$ whenever $X \in \Omega(\sigma, \join^\varepsilon(\mu_\A, \mu_\B))$. Then
%		\[ \PP(X \ne Y) = \PP(X \not\in \Omega(\sigma, \join^\varepsilon(\mu_\A, \mu_\B))) = 1 - p_t(\Omega(\sigma, \join^\varepsilon(\mu_\A, \mu_\B))) < \varepsilon. \]
%	The lemma now gives
	By Lemma \ref{lem:fano},
	\begin{align*}
		&\hspace{-1cm} \log \abs{\Omega(\sigma, \join^\varepsilon(\mu_\A, \mu_\B))} \\
			%&\geq \shent(Y) \\
			%&\geq \shent(X,Y) - [ 1 + \PP(X \ne Y) \log \abs{\A \times \B}^{V} ]\\
			&\geq \shent(p_t) - [\log 2+\varepsilon \abs{V} \log \abs{\A \times \B}] \\
			&\geq \shent(p_0) + c \abs{V} t - [5 \varepsilon \abs{V} ( \abs{U^\A_e}_{\Lip} + \abs{U^\B_e}_{\Lip} )] - [\log 2+\varepsilon \abs{V} \log \abs{\A \times \B}] \\
			&= \log \abs{\Omega(\sigma, \calO)} + \abs{V}\big(ct - \varepsilon[5\abs{U^\A_e}_{\Lip} + 5\abs{U^\B_e}_{\Lip}+ \log \abs{\A \times \B}]\big) - \log 2.
	\end{align*}
	Since $\Omega(\sigma, \join^\varepsilon(\mu_\A, \mu_\B)) \subset \Omega(\sigma, \join^\eta(\mu_\A, \mu_\B))$, exponentiating both sides gives the claimed inequality with $C_1 = ct$ and $C_2 = 5\abs{U^\A_e}_{\Lip} + 5\abs{U^\B_e}_{\Lip}+ \log \abs{\A \times \B}$.
\end{proof}

\begin{proof}[Proof of Theorem \ref{thm:gibbsmaximal}]
Suppose $\lambda$ is not Gibbs, and pick $\varepsilon, \eta > 0$. By Proposition \ref{prop:modelsbound} we can pick $\delta > 0$ and $\calO \ni \lambda$ such that if $\sigma \in \Hom(\Gamma, \Sym(V))$ satisfies $\Delta^\sigma < \delta$ then
		\[ \abs{\Omega(\sigma, \join^\eta(\mu_\A, \mu_\B) )} \geq \abs{\Omega(\sigma, \calO)} \cdot \tfrac{1}{2} \exp\!\big[ \abs{V} (C_1 - \varepsilon C_2) \big] . \]

Since the probability that $\Delta^{\sigma_n} < \delta$ approaches 1 superexponentially fast in $n$, this implies
	\begin{align*}
		\limsup_{n \to \infty} \frac{1}{\abs{V_n}} \log \EE \abs{\Omega(\sigma_n, \join^\eta(\mu_\A, \mu_\B))}
			&\geq \limsup_{n \to\infty} \frac{1}{\abs{V_n}} \log \EE \abs{\Omega(\sigma_n, \calO)} + (C_1 - \varepsilon C_2) \\
			&\geq \inf_{\calO \ni \lambda} \limsup_{n \to\infty} \frac{1}{\abs{V_n}} \log \EE \abs{\Omega(\sigma_n, \calO)} + C_1 - \varepsilon C_2 \\
			&= \h_\Sigma(\lambda) + C_1 - \varepsilon C_2
	\end{align*}
	Since $\varepsilon>0$ was arbitrary,
		\[ \limsup_{n \to \infty} \frac{1}{\abs{V_n}} \log \EE \abs{\Omega(\sigma_n, \join^\eta(\mu_\A, \mu_\B))}
			\geq \h_\Sigma(\lambda) + C_1 .\]
	Taking the infimum over $\eta>0$ gives
		\[  \h_\Sigma(\lambda) + C_1 \leq \inf_{\eta>0} \limsup_{n \to \infty} \frac{1}{\abs{V_n}} \log \EE \abs*{\Omega(\sigma_n, \join^\eta(\mu_\A, \mu_\B))} . \]
	The remainder of the proof is analogous to the proof of \cite[Theorem C]{shriver2020}.
	
	By compactness, we can let $\calF \subset \join^\eta(\mu_\A, \mu_\B)$ be a finite set with $\join^\eta(\mu_\A, \mu_\B) \subset \bigcup_{\theta \in \calF} \ball[\bar{d}]{\theta}{\eta}$. Then
	\begin{align*}
		\h_\Sigma(\lambda) + C_1
			&\leq \limsup_{n \to \infty} \frac{1}{\abs{V_n}} \log \EE \abs*{\Omega\left(\sigma_n,\ \bigcup_{\theta \in \calF} \ball{\theta}{\eta}\right)} \\
			 &= \max_{\theta \in \calF} \limsup_{n \to \infty} \frac{1}{\abs{V_n}} \log \EE\abs*{\Omega\left(\sigma_n, \ball{\theta}{\eta}\right)} \\
			 &\leq \sup_{\theta \in \join^\eta} \limsup_{n \to \infty} \frac{1}{\abs{V_n}} \log \EE\abs*{\Omega\left(\sigma_n, \ball{\theta}{\eta}\right)} .
	\end{align*}
	Now for each $m \in \NN$ take $\eta = 1/m$, and let $\theta_m \in \join^\eta(\mu_\A, \mu_\B)$ get within $1/m$ of the supremum in the last line of the previous display. By compactness, we can pass to a weakly-convergent subsequence $\theta_{m_k}$ with limit $\theta_\infty$, which must lie in $\join(\mu_\A, \mu_\B)$. Given $\calO \ni \theta_\infty$, for $m$ large enough we have $\ball{\theta_m}{1/m} \subset \calO$. Therefore
	\begin{align*}
		\h_\Sigma(\lambda) + C_1
			&\leq \limsup_{n \to \infty} \frac{1}{\abs{V_n}} \log \EE \abs*{\Omega\left(\sigma_n, \ball{\theta_m}{1/m}\right)} + \tfrac{1}{m} \\
			&\leq \limsup_{n \to \infty} \frac{1}{\abs{V_n}} \log \EE\abs*{\Omega\left(\sigma_n, \calO \right)} + \tfrac{1}{m}.
	\end{align*}
	Taking $m$ to infinity then the infimum over $\calO$ gives
		\[ \h_\Sigma(\lambda) + C_1 \leq \h_\Sigma(\theta_\infty) . \]
	Since $C_1 > 0$ and $\theta_\infty \in \join(\mu_\A, \mu_\B)$, this means that $\h_\Sigma(\lambda)$ is not maximal, unless every joining has $\h_\Sigma = -\infty$.
\end{proof}

In some cases, this allows us to say exactly which measure maximizes $\h_\Sigma$:

\begin{cor}
\label{cor:prodmax}
	Suppose $\mu_\A \in \ex \Gibbs(\Phi^\A) \cap \Prob^\Gamma(\A^\Gamma)$ and $\mu_\B \in \Gibbs^\Gamma(\Phi^\B)$. Then
		\[ \Gibbs(\Phi^\A \oplus \Phi^\B) \cap \join(\mu_\A, \mu_\B) = \{ \mu_\A \times \mu_\B\} . \]
	In particular,
		\[ \sup_{\lambda \in \join(\mu_\A, \mu_\B)} \h_\Sigma(\lambda) = \h_\Sigma(\mu_\A \times \mu_\B). \]
\end{cor}
Note that we require $\mu_\A$ to be an extreme point of the set of all Gibbs measures, not just the shift-invariant ones.
\begin{proof}
	By \cite[Equation (7.19)]{georgii2011}, we have
	\[ \ex \Gibbs(\Phi) = \{ \mu \times \nu \st \mu \in \ex \Gibbs(\Phi^\A),\ \nu \in \ex\Gibbs(\Phi^\B) \} . \]
Let $\lambda$ be a joining of $\mu_\A, \mu_\B$ which is in $\Gibbs(\Phi)$, and write its extreme decomposition in $\Gibbs(\Phi)$ as 
	\[ \lambda = \int \mu \times \nu\, \xi(d\mu, d\nu) . \]
Then taking the marginal on $\A^\Gamma$ gives
	\[ \mu_\A = \int \mu \, \xi(d\mu, d\nu) , \]
so extremality of $\mu_\A$ implies that $\xi$ gives full mass to the set $\{ (\mu_\A, \nu) \st \nu \in \ex \Gibbs(\Phi^\B)\}$. Therefore
	\[ \lambda = \mu_\A \times \int \nu\, \xi(d\mu, d\nu) = \mu_\A \times \mu_\B . \qedhere \]
\end{proof}

For example, at and above the reconstruction threshold, the free-boundary Ising Gibbs measure $\mu^{FB}$ is extreme \cite{bleher1995,ioffe1996}. Therefore given any other fixed Gibbs measure (possibly for another nearest-neighbor potential and temperature), the product joining with $\mu^{FB}$ has maximal $\h_\Sigma$. \\

%\begin{lemma}
%	For any $\mu \in \Prob^\Gamma(\A^\Gamma)$, and any random sofic approximation $\Sigma$,
%		\[ \h_\Sigma(\mu \diag \mu) = \h_\Sigma(\mu). \]
%\end{lemma}

We also note the following corollary:

\begin{cor}
\label{cor:nonzero}
	If $\mu \in \Prob^\Gamma(\A^\Gamma)$ is a Gibbs measure and $\abs{\A} > 1$, then for every deterministic sofic approximation $\Sigma$ we have $\h_\Sigma(\mu) \neq 0$.
\end{cor}
Since for deterministic sofic approximations we always have $\h_\Sigma(\mu) \in \{-\infty\} \cup [0, +\infty)$ we could also write the conclusion as ``either $\h_\Sigma(\mu) = -\infty$ or $\h_\Sigma(\mu) > 0$.'' Informally, we could then say that any deterministic sofic approximation either supports no good models for $\mu$ at all, or else the number of good models has a strictly positive (upper) exponential growth rate.
\begin{proof}
	Suppose $\h_\Sigma(\mu) \ne -\infty$. Since the diagonal self-joining $\mu \diag \mu$ is not Gibbs, Theorem \ref{thm:gibbsmaximal} implies the existence of some other self-joining $\lambda$ with $\h_\Sigma(\lambda) > \h_\Sigma(\mu \diag \mu)$. But then
		\[ \h_\Sigma(\mu) = \h_\Sigma(\mu \diag \mu) < \h_\Sigma(\lambda) \leq 2 \h_\Sigma(\mu) , \]
	where the last inequality depends on $\Sigma$ being deterministic.
\end{proof}

\section{Non-optimal Gibbs joinings}
\label{sec:nontriv}
One might wonder whether the converse of Theorem \ref{thm:gibbsmaximal} is true: does every joining of two Gibbs measures which is Gibbs for their sum interaction maximize entropy?

In this section we restrict to a particular random sofic approximation: Assume that $\Gamma$ is the rank-$r$ free group, and let $\sigma_n \in \Hom(\Gamma, \Sym([n]))$ be uniformly random. The paper \cite{bowen2010} shows that $h_\Sigma$ is the $\f$-invariant introduced in \cite{bowen2010a}; see also the survey \cite{bowen2020a} for more information on the $\f$-invariant.

A particularly useful property, not shared by all variants of sofic entropy, is additivity: $\f(\mu \times \nu) = \f(\mu) + \f(\nu)$.

We also restrict to a particular class of Gibbs measures: the (free boundary conditions) Ising measure with transition probability $\varepsilon \in (0,1/2]$ is the $\Gamma$-indexed, $\{-1,+1\}$-valued stationary Markov chain with uniform single-vertex marginals and transition matrix
	\[ \begin{pmatrix}
		1-\varepsilon & \varepsilon \\
		\varepsilon & 1-\varepsilon
	\end{pmatrix} . \]
We denote the distribution by $\Is{\varepsilon} \in \Prob^\Gamma(\{\pm 1\}^\Gamma)$. For each $\varepsilon$, the measure $\Is{\varepsilon}$ is Gibbs for the nearest-neighbor interaction with $h \equiv 0$ and $J(\ta, \tb) = - \beta \ta\tb$, where the ``inverse temperature'' $\beta$ is determined by the relation
	\[ \frac{\varepsilon}{1-\varepsilon} = \exp(-2\beta) . \]
If $\varepsilon$ is small then $\beta$ is large, so we think of this as ``low temperature.''
We can also think of $\Is{\varepsilon}$ as a model for broadcasting information, where we start with a uniformly random bit at the identity and transmit it across edges with error probability $\varepsilon$. \\

Since $\Is{\varepsilon}$ is a Markov chain, its $\f$-invariant can be easily calculated. It is given by
\begin{equation}
\label{eqn:Isingf}
	\f(\Is{\varepsilon}) = \log 2 + r (\shent(\varepsilon) - \log 2)
\end{equation}
where $\shent(\varepsilon) = -[\varepsilon \log \varepsilon + (1-\varepsilon)\log (1-\varepsilon) ]$ \cite[Section 3.3]{bowen2020a}. In particular, $\f(\Is{\varepsilon}) < 0$ for small enough $\varepsilon$. It is also not too difficult to show that if $\Is{\varepsilon} \diag \Is{\varepsilon}$ is the diagonal self-joining then
	\[ \f(\Is{\varepsilon} \diag \Is{\varepsilon}) = \f(\Is{\varepsilon}) . \]
Therefore if $\f(\Is{\varepsilon}) < 0$ then the product joining is not optimal, since $2\f(\Is{\varepsilon}) < \f(\Is{\varepsilon})$. We can extend this to the case $\f(\Is{\varepsilon}) = 0$, since Theorem \ref{thm:gibbsmaximal} implies that the diagonal joining is non-optimal.

This already answers the question posed at the beginning of this section in the negative: the product joining is always Gibbs for the sum interaction, but is not maximal for small enough $\varepsilon$. In the rest of this section we extend further the range of $\varepsilon$ where this is true.

\begin{mainthm}
\label{thm:main3}
	Let
		\[ \varepsilon_c = \frac{1}{2} - \Pstar \frac{1}{\sqrt{2r}} + o_r(r^{-1/2}) , \]
	where $\Pstar \approx 0.7632$.
	If $\varepsilon < \varepsilon_c$ then the product self-joining of $\Is{\varepsilon}$ is non-optimal.
\end{mainthm}

The constant $\Pstar$ is the limiting ground state energy density of the Sherrington-Kirkpatrick model; we will not need its precise definition here.

%Some numerical evidence suggests that this may remain true if $\varepsilon_c$ is replaced by the larger reconstruction threshold, which is the solution to $(1-2\varepsilon)(2r-1)^2 = 1$; as noted above, this is also the exact threshold above which the free-boundary Gibbs state is extreme.

Let $\varepsilon_{\f} < 1/2$ be the smaller solution to $\f(\Is{\varepsilon}) = 0$. If $\varepsilon \leq \varepsilon_{\f}$ then $\f(\Is{\varepsilon}) \leq 0$; we have remarked above that this implies non-optimality of the product joining. A Taylor expansion of $\shent$ yields from Equation~\ref{eqn:Isingf}
	\[ \varepsilon_{\f} = \frac{1}{2} - \sqrt{\log 2} \frac{1}{\sqrt{2r}} + o_{r \to \infty} (r^{-1/2}) . \]
Since $\sqrt{\log 2} \approx 0.8326 > 0.7632 \approx \Pstar$,
	\[ \varepsilon_{\f} < \varepsilon_c \quad\text{for all large } r. \]
Therefore this theorem does, in fact, extend the range of non-maximality of the product (for large enough $r$).

To prove the theorem, we will use a result of \cite{dembo2017} to argue that, for some $\varepsilon$ below the reconstruction threshold but above where the $\f$-invariant is 0, the optimal Ising self-joining is not the product or the diagonal joining. %This disproves a conjecture of Bowen.

We first introduce some relevant terminology. For a finite graph $G = (V,E)$, a bisection is a partition $V = V_1 \sqcup V_2$ where $\abs{V_1} = \abs{V_2}$ if $\abs{V}$ is even, or the sizes differ by 1 if $\abs{V}$ is odd. The cut size of a bipartition $V = V_1 \sqcup V_2$ is the number of edges whose endpoints lie in different parts. The smallest cut size of any bisection of $G$ is denoted $\mcut(G)$. For the graph of $\sigma \in \Hom(\Gamma, \Sym(V))$, we will simply write $\mcut(\sigma)$.

The relevant result we will use is the following:
\begin{theorem}[{modification of \cite[Theorem 1.5]{dembo2017}}]
\label{thm:dms}
	Let $\sigma \in \Hom(\Gamma, \Sym(V))$ be chosen uniformly at random. Then as $\abs{V} \to \infty$,
		\[ \frac{\mcut(\sigma)}{\abs{V}r} \xrightarrow{\PP} \varepsilon_c . \]
\end{theorem}
Here ``$\xrightarrow{\PP}$'' denotes convergence in probability. Note that the existence of some related limits was established earlier in \cite{bayati2013}, but the particular form of the asymptotic $\varepsilon_c$ (found in \cite{dembo2017}) is useful here due to its similarity to $\varepsilon_{\f}$.
\begin{proof}
Let $G^{\mathrm{reg}}(V,d)$ denote a $d$-regular graph with vertex set $V$, chosen uniformly at random (undefined unless $\abs{V}d$ is even). Theorem 1.5 of \cite{dembo2017} states that
		\[ \frac{\mcut(G^{\mathrm{reg}}(V,2r))}{\abs{V}r} \xrightarrow{\PP} \varepsilon_c . \]
	They actually prove the stronger result that this holds when $G^{\mathrm{reg}}(V,d)$ is a random multigraph chosen according to the configuration model.
By the main theorems of \cite{greenhill2002}, the same holds with $G^{\mathrm{reg}}(V,d)$ replaced by a uniformly random $\sigma \in \Hom(\Gamma, \Sym(V))$.
\end{proof}

The connection between the Ising model and $\mcut$ is that if a graph $G$ admits a good model for $\Is{\varepsilon}$, then $\mcut(G)$ must not be much bigger than $\varepsilon\abs{V}r$: since the single-vertex marginal of $\Is{\varepsilon}$ is uniform, this good model must approximately bisect $V$, and since the transition probability is $\varepsilon$, the cut size of the corresponding partition must be approximately $\varepsilon \abs{V}r$ (since $\abs{V}r$ is the total number of edges). More precisely:

\begin{lemma}
\label{lem:isingcut}
	For every $\delta > 0$ there exists a neighborhood $\calO \ni \Is{\varepsilon}$ such that for every $\sigma \in \Hom(\Gamma, \Sym(V))$
		\[ \Omega(\sigma, \calO) \ne \varnothing \quad \Rightarrow \quad \mcut(\sigma) < \abs{V}r(\varepsilon + \delta) . \]
\end{lemma}
\begin{proof}
	Let $\rho = \frac{\delta}{2r+1}$, and let $\calO$ be the set of $\nu \in \Prob^\Gamma(\{\pm 1\}^\Gamma)$ whose marginal on $\ball[\Gamma]{e}{1}$ is within total variation distance $\rho$ of the same marginal of $\Is{\varepsilon}$.
	
	Suppose we have $\mb{x} \in \Omega(\sigma, \calO)$. Then
		\[ \abs*{ \tfrac{1}{\abs{V}} \abs{\{ v \in V \st \mb{x}(v) = +1 \}} - \frac{1}{2}} < \rho \]
	so we can pick $\mb{y} \in \{\pm 1\}^V$ with
		\[ \abs*{\abs{\{v \in V \st \mb{y}(v) = +1\}} - \frac{\abs{V}}{2}} \leq 1 \quad\text{and}\quad  \frac{1}{\abs{V}}\abs{\{v \in V \st \mb{y}(v) \ne \mb{x}(v) \}} < \rho . \]
	Now $\mb{y}$ induces a bisection of $V$, and
	\begin{align*}
		\abs*{ \frac{1}{\abs{V}r} \sum_{v \in V} \sum_{i \in r} \1 \{ \mb{y}(v) \ne \mb{y}(\sigma^i v)\} - \varepsilon }
			&\leq \abs*{ \frac{1}{\abs{V}r} \sum_{v \in V} \sum_{i \in r} \1 \{ \mb{x}(v) \ne \mb{x}(\sigma^i v)\} - \varepsilon }\\
				&\quad + \frac{1}{\abs{V}r} \sum_{v \in V} \sum_{i \in r} \abs*{\1 \{ \mb{y}(v) \ne \mb{y}(\sigma^i v)\} - \1 \{ \mb{x}(v) \ne \mb{x}(\sigma^i v)\} } .
	\end{align*}
	The first term is at most $\rho$ by definition of $\calO$: to see this, write
		\[ \frac{1}{\abs{V}r} \sum_{v \in V} \sum_{i \in r} \1 \{ \mb{x}(v) \ne \mb{x}(\sigma^i v)\} = \int \frac{1}{r} \sum_{i \in [r]} \1 \{ \mb{z}(e) \ne \mb{z}(s_i)\}\, P_{\mb{x}}^\sigma(d \mb{z}) . \]
	To bound the second term, write
	\begin{multline*}
		\sum_{v \in V} \sum_{i \in r} \abs*{\1 \{ \mb{y}(v) \ne \mb{y}(\sigma^i v)\} - \1 \{ \mb{x}(v) \ne \mb{x}(\sigma^i v)\} } \\
		\leq \sum_{v \in V} \sum_{i \in r} \big[\1 \{ \mb{y}(v) \ne \mb{x}(v)\} + \1 \{ \mb{y}(\sigma^i v) \ne \mb{x}(\sigma^i v)\} \big] \\
		= 2 r \sum_{v \in V} \1 \{ \mb{y}(v) \ne \mb{x}(v)\}
		\leq 2 r \abs{V} \rho .
	\end{multline*}
	Therefore the cut size of the bisection induced by $\mb{y}$ is at most
		\[ \abs{V}r \varepsilon + \abs{V}r \rho + 2r \abs{V} \rho = \abs{V}r(\varepsilon + \delta) . \qedhere \]
\end{proof}

\subsection{Proof of Theorem \ref{thm:main3}}

Non-optimality of the product joining for $\varepsilon < \varepsilon_c$ follows from the next two lemmas.

\begin{lemma}
	Suppose that $\Is{\varepsilon} \times \Is{\varepsilon}$ has maximal $\f$ among all self-joinings of $\Is{\varepsilon}$. Then for any $\delta > 0$
		\[ \liminf_{n \to \infty} \frac{1}{n} \log \PP\big(\tfrac{\mcut(\sigma_n)}{rn} < \varepsilon + \delta\big) \geq 0 . \]
\end{lemma}
\begin{proof}
	A standard argument shows that
		\[ \inf_{\calO \ni \Is{\varepsilon}} \limsup_{n \to \infty} \frac{1}{n} \log \EE \big[ \abs{\Omega(\sigma_n, \calO)}^2 \big] = \sup_{\lambda \in \join(\Is{\varepsilon},\Is{\varepsilon})} \f(\lambda) = 2 \f(\Is{\varepsilon}) , \]
	where the second equality uses our assumption that the product joining is optimal. Therefore for any $\eta$, for all small enough $\calO$ we have
		\[ \EE \big[\abs{\Omega(\sigma_n, \calO)}^2 \big] < \exp\big[ n ( 2 \f(\Is{\varepsilon}) + \eta ) \big] \]
	for all large enough $n$.
	Similarly, since $\f(\Is{\varepsilon}) = \inf_{ \calO \ni \Is{\varepsilon}} \limsup_{n \to \infty} \frac{1}{n} \log \EE \abs{\Omega(\sigma_n, \calO)}$, for any $\calO \ni \Is{\varepsilon}$ we have
		\[ \EE \abs{\Omega(\sigma_n, \calO)} > \exp\big[ n ( \f(\Is{\varepsilon}) - \eta) \big] \]
	for infinitely many $n$.
	
	By Lemma~\ref{lem:isingcut}, for all small enough $\calO \ni \Is{\varepsilon}$ we have
		\[ \PP(\tfrac{\mcut(\sigma_n)}{rn} < \varepsilon + \delta) \leq \PP( \Omega(\sigma_n, \calO) \ne \varnothing ) . \]
	Using the Paley-Zygmund inequality,
	\begin{align*}
		\PP( \Omega(\sigma_n, \calO) \ne \varnothing )
			&\geq \PP\big( \abs{\Omega(\sigma_n, \calO)} > \tfrac{1}{2} \EE \abs{\Omega(\sigma_n, \calO)} \big) \\
			&\geq (1 - \tfrac{1}{2})^2 \frac{\big[\EE |\Omega(\sigma_n, \calO)|]^2}{\EE \big[\abs{\Omega(\sigma_n, \calO)}^2\big]} \\
			&> \tfrac{1}{4} \exp\big[ -2\eta n\big]
	\end{align*}
	for infinitely many $n$. Hence
		\[ \liminf_{n \to \infty} \frac{1}{n} \log \PP\big(\tfrac{\mcut(\sigma_n)}{rn} < \varepsilon + \delta\big) > -2 \eta \]
	and, since $\eta>0$ is arbitrary, the result follows.
\end{proof}

\begin{lemma}
\label{lem:liminfneg}
	If $\varepsilon < \varepsilon_c$ then for all small enough $\delta > 0$
	\[
		\liminf_{n \to \infty} \frac{1}{n} \log \PP(\tfrac{\mcut(\sigma_n)}{rn} < \varepsilon + \delta) < 0 .
	\]
\end{lemma}
\begin{proof}
	Theorem \ref{thm:dms} implies that $\lim_{n \to \infty} \EE\tfrac{\mcut(\sigma_n)}{rn} = \varepsilon_c$, so if $\delta$ is small enough that $\varepsilon+\delta < \varepsilon_c$ then for any $0 < t <  \varepsilon_c - (\varepsilon+\delta)$
	\begin{align*}
		\PP(\tfrac{\mcut(\sigma_n)}{rn} < \varepsilon + \delta)
			&= \PP(\tfrac{\mcut(\sigma_n)}{rn} - \EE\tfrac{\mcut(\sigma_n)}{rn} < \varepsilon + \delta - \EE\tfrac{\mcut(\sigma_n)}{rn}) \\
			&\leq \PP(\tfrac{\mcut(\sigma_n)}{rn} - \EE\tfrac{\mcut(\sigma_n)}{rn} < -t) \tag{for all large $n$} \\
			&\leq \PP(\abs*{\tfrac{\mcut(\sigma_n)}{rn} - \EE\tfrac{\mcut(\sigma_n)}{rn} } \geq t) .
	\end{align*}
	By a standard ``switching'' argument (Lemma~\ref{lem:concentration}), we have
		\[ \PP( \abs{\tfrac{\mcut(\sigma_n)}{rn} - \EE\tfrac{\mcut(\sigma_n)}{rn}} \geq t ) \leq 2 \exp(-t^2 n r/8) \quad \forall t > 0 . \]
	The result follows.
\end{proof}

\subsection{Concentration}

Here we develop an analogue of \cite[Theorem 2.19]{wormald1999}, which proves exponential concentration for functions which are not changed much under ``switching.'' Similar concentration techniques also appear in the survey \cite{mcdiarmid1998}.

Given $\tau_1, \tau_2 \in \Sym(n)$, we write $\tau_1 \sim \tau_2$ if 
	\[ \abs{\{j \in [n] \st \tau_1(j) \ne \tau_2(j) \}} = 2 . \]
Note that 2 is the smallest positive number of disagreements between two permutations. If $\tau_1 \sim \tau_2$ and $i,j \in [n]$ are the points where they disagree, then it must be that $\tau_1(i) = \tau_2(j)$ and $\tau_2(i) = \tau_1(j)$. For this reason we say they differ by a \emph{switching}.

We extend this to homomorphisms $\sigma_1, \sigma_2 \colon \FF_r \to \Sym(n)$ by saying $\sigma_1 \sim \sigma_2$ whenever there is exactly one $i_0 \in [r]$ with $\sigma_1^{i_0} \sim \sigma_2^{i_0}$ and for all $i \ne i_0$ we have $\sigma_1^i = \sigma_2^i$.

If $\sigma_1 \sim \sigma_2$ then $\abs{\mcut(\sigma_1) - \mcut(\sigma_2)} \leq 2$. The following lemma establishes concentration for functions with this property.

\begin{lemma}
\label{lem:concentration}
	Suppose $g$ is a real-valued function on $\Hom(\FF_r, \Sym(n))$ such that $\abs{g(\sigma_1) - g(\sigma_2)} \leq c$ whenever $\sigma_1 \sim \sigma_2$. Then if $\sigma$ is chosen uniformly at random
		\[ \PP \big( \abs{g(\sigma) - \EE g(\sigma)} > t \big) \leq 2 \exp\left( \frac{-t^2}{2nrc^2} \right) . \]
\end{lemma}
\begin{proof}
	We choose $\sigma$ by picking $\sigma^i(j)$ in lexicographic order on $(i,j) \in [r] \times [n]$ uniformly from all allowable choices. Let
		\[ \{\varnothing, \Hom(\FF_r, \Sym(n))\} = \calF_0 \subseteq \calF_1 \subseteq \cdots \subseteq \calF_{nr} = \calP(\Hom(\FF_r, \Sym(n))) \]
	be the filtration induced by these choices. If we show that
		\[ \abs*{\EE[ g(\sigma) \mid \calF_k] - \EE[ g(\sigma) \mid \calF_{k-1}]} \leq c \quad \text{for all } k, \]
	then the result will follow from Azuma-Hoeffding.
	
	Fix $k = i_0 r + j_0 \in [nr]$, so that $\calF_k$ records the choice of $\sigma^{i_0}(j_0)$ and all previous choices. It is helpful to think of, for $\sigma_0 \in \Hom(\Gamma, \Sym(n))$,
	\begin{align*}
		\EE[ g(\sigma) \mid \calF_k](\sigma_0) &= \EE[ g(\sigma) \mid \sigma^i(j) = \sigma_0^i(j) \ \forall (i,j) \leq (i_0, j_0)] \\
		\EE[ g(\sigma) \mid \calF_{k-1}](\sigma_0) &= \EE[ g(\sigma) \mid \sigma^i(j) = \sigma_0^i(j) \ \forall (i,j) < (i_0, j_0)] .
	\end{align*}
	We need to show that the difference between these two quantities is bounded by $c$ for each fixed $\sigma_0$.
	
	Let $A \subset [n] $ be the set of allowed values for $\sigma^i(j)$ given the event $U \coloneqq \{\sigma^i(j) = \sigma_0^i(j) \ \forall (i,j) < (i_0, j_0)\}$. For each $a \in A$ let $U_a = U \cap \{\sigma^i(j) = a\}$. Note that each $U_a$ has the same probability, namely $\frac{1}{\abs{A}} \PP(U)$. For convenience write $a_0 = \sigma_0^i(j)$. Then we can rewrite the above quantities as
		\[ \EE[ g(\sigma) \mid U_{a_0}] \quad \text{and} \quad \EE[ g(\sigma) \mid U] . \]
	Then
	\begin{align*}
		\abs*{\EE[ g(\sigma) \mid U_{a_0}] -  \EE[ g(\sigma) \mid U]}
			&= \abs*{\EE[ g(\sigma) \mid U_{a_0}] - \frac{1}{\abs{A}} \sum_{a \in A} \EE[ g(\sigma) \mid U_a]} \\
			&\leq \frac{1}{\abs{A}} \sum_{a \in A} \abs[\Big]{\EE[ g(\sigma) \mid U_{a_0}] - \EE[ g(\sigma) \mid U_a]}
	\end{align*}
	For $\sigma \in U$ and $a \in A$, let $S_a \sigma$ denote the unique switching of $\sigma$ with $(S_a \sigma)^i(j) = a$ (or take $S_a \sigma = \sigma$ if $\sigma \in U_a$ already). Note that $\sigma \in U$ implies $S_a \sigma \in U_a$. Moreover, if $\sigma \sim \Unif(U_{a_0})$ then $S_a \sigma \sim \Unif(U_a)$ (since $S_a$ is a bijection). Therefore
	\begin{align*}
		\abs[\Big]{\EE[ g(\sigma) \mid U_{a_0}] - \EE[ g(\sigma) \mid U_a]}
			&\leq \abs[\Big]{\EE[ g(S_a\sigma) \mid U_{a_0}] - \EE[ g(\sigma) \mid U_{a_0}]} + \abs[\Big]{\EE[ g(S_a \sigma) \mid U_{a_0}] - \EE[ g(\sigma) \mid U_a]} \\
			&\leq c + 0,
	\end{align*}
	so the result follows.
\end{proof}

\printbibliography

\end{document}